\documentclass{amsart}

\usepackage{amsmath, amssymb, amsthm, mathtools, tikz-cd}
\usepackage[hidelinks]{hyperref}

\theoremstyle{plain}
\newtheorem{theorem}{Theorem}
\newtheorem{proposition}[theorem]{Proposition}
\newtheorem{lemma}[theorem]{Lemma}
\newtheorem{corollary}[theorem]{Corollary}

\theoremstyle{definition}
\newtheorem{definition}[theorem]{Definition}
\newtheorem{remark}[theorem]{Remark}

\newcommand{\id}{\mathrm{id}}

\usepackage{tikz,xcolor,mathtools}
\usetikzlibrary{calc, decorations.markings, arrows.meta, patterns, patterns.meta, backgrounds}

\makeatletter
\tikzset{nomorepostaction/.code=\let\tikz@postactions\pgfutil@empty}
\makeatother

\DeclareRobustCommand{\SkipTocEntry}[5]{}

\usepackage{eqparbox}
\newcommand{\eqrel}[2][B]{\mathrel{\eqmakebox[#1]{#2}}}

\title{An asymptotic homotopy lifting property}
\author[J. Carri\'on]{Jos\'e R.\ Carri\'on}
\address{Jos\'e R.\ Carri\'on, Department of Mathematics, Texas Christian
	University, Fort Worth, Texas 76129, USA}
\email{j.carrion@tcu.edu}

\author[C. Schafhauser]{Christopher Schafhauser}
\address{Christopher Schafhauser, Department of  Mathematics, University of Nebraska - Lincoln, Lincoln, Nebraska, USA}
\email{cschafhauser2@unl.edu}

\dedicatory{In memory of Eberhard Kirchberg}

\date{\today}

\begin{document}

\begin{abstract}
  A $C^*$-algebra $A$ is said to have the \emph{homotopy lifting property} if for all $C^*$-algebras $B$ and $E$, for every surjective $^*$-homomorphism $\pi \colon E \rightarrow B$ and for every $^*$-homomorphism $\phi \colon A \rightarrow E$, any path of $^*$-homomorphisms $A \rightarrow B$ starting at $\pi \phi$ lifts to a path of $^*$-homomorphisms $A \rightarrow E$ starting at $\phi$.
  Blackadar has shown that this property holds for all semiprojective $C^*$-algebras.
	
  We show that a version of the homotopy lifting property for asymptotic morphisms holds for separable $C^*$-algebras that are sequential inductive limits of se\-mi\-pro\-ject\-ive $C^*$-algebras.
  It also holds for any separable $C^*$-algebra if the quotient map $\pi$ satisfies an approximate decomposition property in the spirit of (but weaker than) the notion of quasidiagonality for extensions.
\end{abstract}

\maketitle

\addtocontents{toc}{\SkipTocEntry}
\section*{Introduction}\label{sec:introduction}
\renewcommand{\thetheorem}{\Alph{theorem}}
	
A pair $(A, \pi)$ consisting of a $C^*$-algebra $A$ and a surjective $^*$-homomorphism $\pi \colon E \rightarrow B$ between $C^*$-algebras has the \emph{homotopy lifting property} if
we have the following path-lifting property along $\pi$: if  $\phi \colon A \rightarrow B$ is a $^*$-homomorphism that lifts to a  $^*$-homomorphism $\tilde \phi \colon A \rightarrow E$, then any continuous path $\theta$ of $^*$-homomorphisms $A \rightarrow B$ starting at $\phi$ lifts to a continuous path $\tilde\theta$ of $^*$-homomorphisms $A \rightarrow E$ starting at $\tilde \phi$.
In other words,
the diagram completion problem
\begin{equation}\label{eq:intro-diagram}
\begin{tikzcd}
  A \arrow[dashed]{dr}{\tilde\theta}
    \arrow[bend right]{ddr}[swap]{\theta}
    \arrow[bend left]{drr}{\tilde\phi}
  &[-10pt] &[10pt] \\[-5pt] 
  & IE \arrow{d}[swap]{I\pi} \arrow{r}{{\rm ev}_0^E} & E
    \arrow{d}{\pi} \\[10pt] 
  & IB \arrow{r}[swap]{{\rm ev}_0^B} & B 
\end{tikzcd}
\end{equation}
always has a solution, where we write $ID = C([0, 1], D)$ for any $C^*$-algebra $D$.

The homotopy lifting property and its dual \emph{homotopy extension property} play a fundamental role in homotopy theory.
A classical result of Borsuk \cite{Borsuk37} (in its dual form) states that if $X$ is a (compact Hausdorff) absolute neighborhood retract, then $(C(X), \pi)$ satisfies a restricted form of the homotopy lifting property where one requires the $C^*$-algebras $B$ and $E$ above to be commutative.
This holds, for example, when $X$ is a finite CW-complex.

The analogs of absolute neighborhood retracts in the noncommutative setting are the semiprojective $C^*$-algebras.
Blackadar proved an analog of Borsuk's homotopy extension theorem in \cite{Blackadar16}, showing that if $A$ is a semiprojective $C^*$-algebra, then $(A, \pi)$ has the homotopy lifting property for all $\pi$.

Most ``naturally occurring'' noncommutative $C^*$-algebras (such as simple $C^*$-algebras) are highly singular from a topological point of view, due to the abundance of inner automorphisms, yielding intricate internal dynamics on the $C^*$-algebra.
Consequently, there are typically very few maps (i.e., $^*$-homomorphisms) between two non-commutative $C^*$-algebras.
There are several natural ways of enlarging the collection of morphisms between $C^*$-algebras.
One common approach uses the theory of asymptotic morphisms of Connes and Higson \cite{Connes-Higson90}, which was central in the development of $E$-theory.
If $A$ and $B$ are $C^*$-algebras with $A$ separable, an \emph{asymptotic morphism} $\phi \colon A \xrightarrow\approx B$ is a point-norm continuous family of functions $\phi_t \colon A \rightarrow B$, indexed by $t \in \mathbb R_+$, that satisfies the properties of a $^*$-homomorphism in the limit as $t \rightarrow \infty$.

The present paper is devoted to an asymptotic version of the homotopy lifting property for $C^*$-algebras.
This stemmed from our recent work in \cite{top-e}, which introduces a topology on $E$-theory.
A pair $(A, \pi)$ consisting of a separable $C^*$-algebra $A$ and a surjective $^*$-homomorphism $\pi \colon E \rightarrow B$ between $C^*$-algebras $B$ and $E$ has the \emph{asymptotic homotopy lifting property} if the diagram completion problem in \eqref{eq:intro-diagram} always has a solution, where each arrow out of $A$ is interpreted as an asymptotic morphism and the diagram is required to commute in the limit as $t \rightarrow \infty$. (See Definition~\ref{def:ahlp}.)

The asymptotic homotopy lifting property is far less restrictive than the classical version.
In fact, we do not know of any pair $(A, \pi)$ for which it fails.
In the positive direction, we obtain the following result.

\begin{theorem}\label{thm:ahlp-asp}
  If $A$ is a (sequential) inductive limit of semiprojective $C^*$-algebras and $\pi \colon E \rightarrow B$ is a surjective $^*$-homomorphism between $C^*$-algebras $B$ and $E$, then $(A, \pi)$ satisfies the asymptotic homotopy lifting property.
\end{theorem}

A long-standing open question of Blackadar asks if every separable $C^*$-algebra is an inductive limit of semiprojective $C^*$-algebras.
Part of the difficulty of this question arises from the fact that there is no clear obstruction to the existence of such an inductive limit decomposition.
For instance, it is known that every separable $C^*$-algebra is an inductive limit of $C^*$-algebras with semiprojective connecting maps.
Since most properties of semiprojective $C^*$-algebras admit relative versions for semiprojective morphisms, proving that a property holds for inductive limits of semiprojective $C^*$-algebras but does not hold on all separable $C^*$-algebras may be difficult.
Theorem~\ref{thm:ahlp-asp} provides one possible obstruction.

The problem of deciding if a given $C^*$-algebra is an inductive limit of semiprojective $C^*$-algebras can be challenging---for instance, the question is open for $C(S^2)$.
There are, however, many naturally occurring examples of such $C^*$-algebras.
For instance, all AF algebras and all A$\mathbb T$ algebras are inductive limits of semiprojective $C^*$-algebras; this uses the semiprojectivity of $F$ and $F \otimes C(\mathbb T)$ for any finite dimensional $C^*$-algebra $F$ by \cite[Corollary~2.30]{Blackadar85}.
It then follows from \cite{Elliott-Evans93} that irrational rotation algebras are inductive limits of semiprojective $C^*$-algebras.
Graph $C^*$-algebras associated to countable direct graphs form another natural class of examples---indeed, $C^*$-algebras generated by finite graphs are semiprojective (see \cite[Propositions~2.18 and~2.23]{Blackadar85}), and the inductive limit decomposition then follows from the construction in \cite[Section~1]{Raeburn-Szymanski04}.
Further, as all Kirchberg algebras in the UCT class are inductive limits of Kirchberg algebras in the UCT class with finitely generated $K$-theory (see \cite[Proposition~8.4.13]{Rordam02}, for example), a theorem of Enders in \cite{Enders15} (see \cite{Blackadar04, Szymanski02, Spielberg09} for earlier results) implies that Theorem~\ref{thm:ahlp-asp} applies to all UCT Kirchberg algebras $A$.
Several more examples and permanence properties have been obtained by Thiel in \cite{Thiel19a,Thiel19b}.

We also show that the asymptotic homotopy lifting property holds for all separable $C^*$-algebras with some restrictions on the quotient map $\pi$.
Recall that an extension $0 \rightarrow \ker(\pi) \rightarrow E \xrightarrow\pi B \rightarrow 0$ is \emph{quasidiagonal} if $\ker(\pi)$ admits an approximate unit of projections that is quasicentral in $E$.

\begin{theorem}[cf.\ Corollary~\ref{cor:ahlp-qdl}]\label{thm:ahlp-qd}
	If $A$ is a separable $C^*$-algebra and 
	\begin{equation}
		0 \longrightarrow \ker(\pi) \longrightarrow E \overset\pi\longrightarrow B \longrightarrow 0
	\end{equation} 
	is a quasidiagonal extension of $C^*$-algebras, then $(A, \pi)$ satisfies the asymptotic homotopy lifting property.
\end{theorem}

We will prove Theorem~\ref{thm:ahlp-qd} under a condition weaker than quasidiagonality of the extenstion (but equivalent to it in the unital case) that we call \emph{approximate decomposability} (see Definition~\ref{def:qdelicious}).
The property will be defined in terms of the existence of a certain approximate splitting of the quotient map $\pi$.
It is characterized by the existence of an approximate decomposition of $E$ into the direct sum of $\ker(\pi)$ and $B$ in a way that does not require the existence of projections in $E$.

Theorems~\ref{thm:ahlp-asp} and~\ref{thm:ahlp-qd} will be proved simultaneously as corollaries of Theorem~\ref{thm:ahlp}.
The basic idea is to exploit a relative notion of the homotopy lifting property for a pair of $^*$-homomorphisms $(\alpha \colon A_0 \rightarrow A, \pi \colon E \rightarrow B)$, with $\pi$ surjective.
This notion is introduced in Section~\ref{sec:hlp}.
If $A$ is semiprojective, then $(\alpha, \pi)$ will satisfy the homotopy lifting property by Blackadar's previously mentioned result.
Further, when $\alpha$ is semiprojective and $\pi$ is approximately decomposable, the pair $(\alpha, \pi)$ has the homotopy lifting property (see Theorem~\ref{thm:hlp-qdl}).
Using the shape theoretic methods of \cite{Dadarlat94}, we prove the asymptotic homotopy lifting property by writing $A$ as an inductive limit of semiprojective $^*$-homomorphisms $\alpha_n \colon A_n \rightarrow A_{n+1}$ and making use of the homotopy lifting property for each pair $(\alpha_n, \pi)$---see Theorem~\ref{thm:ahlp} for the precise statement.

\addtocontents{toc}{\SkipTocEntry}
\subsection*{Notation}
For the most part, our notation is standard and any non-standard notation is explained in the body as it appears.
However, we take a moment to set out our conventions for algebras of continuous functions, which will be used frequently in the paper.

For a compact Hausdorff space $X$ and a $C^*$-algebra $A$, define $XA = C(X, A)$; in the case $X$ is only locally compact, we use the more standard notation $C_b(X, A)$ (respectively $C_0(X, A)$), for the $C^*$-algebras of continuous bounded functions $X \rightarrow A$ (respectively, those vanishing at infinity).
In the former case, this is most often used when $X$ is the space $I = [0, 1]$.
Other topological spaces that occur frequently are $\mathbb R_+ = [0, \infty)$ and $\mathbb R_{\geq t} = [t, \infty)$ for $t \in \mathbb R_+$.

We write $\mathrm{ev}^A_x \colon XA \rightarrow A$ for the evaluation map $f \mapsto f(x)$.
For a $^*$-ho\-mo\-mor\-phism $\phi \colon A \rightarrow B$ and a compact Hausdorff space $X$, we write $X\phi \colon XA \rightarrow XB$ for the induced $^*$-homomorphism $(X\phi)(f)(x) = \phi(f(x))$.
The case when $X$ is locally compact will occur less often; the analogously defined maps $C_b(X, A) \rightarrow C_b(X, B)$ and $C_0(X, A) \rightarrow C_0(X, B)$ will be denoted by $\bar\phi$.
This notation will be recalled whenever it appears.

\addtocontents{toc}{\SkipTocEntry}
\subsection*{Acknowledgments}
Parts of this project were completed during the first author's visit to University of Nebraska--Lincoln and the second author's visit to Texas Christian University.
We are grateful to the respective universities for their hospitality during these visits.
The second author was partially supported by NSF Grant DMS-2000129.

\renewcommand{\thetheorem}{\arabic{theorem}}
\numberwithin{theorem}{section}
\numberwithin{equation}{section}

\tableofcontents

\section{Approximate and asymptotic morphisms}\label{sec:approx}

This section lays out the definitions of and our conventions on approximate and asymptotic morphisms and collects some preliminary results needed in later sections.
We also recall Dadarlat's homotopy limit construction from \cite{Dadarlat94} for constructing asymptotic morphisms from what we call ``diagrammatic representations'' (see Definition~\ref{def:diag-rep}).

In the literature, approximate morphisms are most often indexed by the natural numbers, but it will be convenient to allow more general index sets.
Note that there is a difference between approximate morphisms and asymptotic morphisms (see Definition~\ref{def:asymp-morph}), even when the index set is $\mathbb R_+$: the topology on $\mathbb R_+$ is accounted for in the definition of asymptotic morphisms, whereas approximate morphisms should be regarded as being indexed over a discrete space.

\begin{definition}\label{def:approx-morph}
  For $C^*$-algebras $A$ and $B$, a net $(\sigma_\lambda\colon A\to B)$ of functions is an \emph{approximate morphism} if, for all $a, a_1, a_2 \in A$ and $r_1, r_2 \in \mathbb C$,
  \begin{enumerate}
  \item $\lim_\lambda \|\sigma_\lambda(r_1a_1 + r_2a_2) - (r_1\sigma_\lambda(a_1) + r_2\sigma_\lambda(a_2))\| = 0$,
  \item $\lim_\lambda \|\sigma_\lambda(a^*) - \sigma_\lambda(a)^* \| = 0$, and
  \item $\lim_\lambda \|\sigma_\lambda(a_1a_2) - \sigma_\lambda(a_1) \sigma_\lambda(a_2)\| = 0$.
  \end{enumerate}
  We will often write $(\sigma_\lambda)\colon A\to B$ for such an approximate morphism.
  We say that $(\sigma_\lambda)$ is
  \begin{itemize}
  \item \emph{linear} if each $\sigma_\lambda$ is linear,
  \item \emph{self-adjoint} if $\sigma_\lambda(a^*) = \sigma_\lambda(a)^*$ for all $\lambda$ and $a \in A$,
  \item \emph{pointwise-bounded} if $\sup_\lambda \|\sigma_\lambda(a)\| < \infty$ for all $a\in A$, and
  \item \emph{equicontinuous} if for all $a_0 \in A$ and $\epsilon > 0$, there is $\delta > 0$ such that for all $a \in A$ with $\|a - a_0\| < \delta$ and all $\lambda$, we have $\|\sigma_\lambda(a) - \sigma_\lambda(a_0) \| < \epsilon$.
  \end{itemize}
\end{definition}

For a directed set $\Lambda$, let $\ell^\infty(\Lambda, B)$ be the $C^*$-algebra all of bounded functions $\Lambda \rightarrow B$, and let $c_0(\Lambda, B) \subseteq \ell^\infty(\Lambda, B)$ be the ideal consisting of such functions converging to 0 along $\Lambda$.
Write $q^B_\Lambda \colon \ell^\infty(\Lambda, B) \rightarrow B_\Lambda$ for the quotient map and let $\iota^B_\Lambda \colon B \rightarrow B_\Lambda$ be the embedding given by sending an element to its corresponding constant net.
Note that any pointwise-bounded approximate morphism $(\sigma_\lambda) \colon A \rightarrow B$ indexed by $\Lambda$ induces a $^*$-homomorphism $A \rightarrow B_\Lambda$.
Conversely, given a $^*$-homomorphism $\sigma_\Lambda \colon A \rightarrow B_\Lambda$, any lift to $\ell^\infty(\Lambda, B)$ is a pointwise-bounded approximate morphism.
With some care, the following result allows us to extend this observation to approximate morphisms that are not pointwise-bounded.

\begin{proposition}
  \label{prop:ptwise-bounded}
  If $A$ and $B$ are $C^*$-algebras and $(\sigma_{\lambda})\colon A\to B$ is an approximate morphism, then the following statements hold:
  \begin{enumerate}
  \item\label{ptbd1} $\limsup_\lambda \|\sigma_\lambda(a)\| \leq \|a\|$ for all $a\in A$;
  \item\label{ptbd2} there is an approximate morphism $(\sigma_\lambda')\colon A\to B$, indexed by the same directed set, that is self-adjoint, linear, pointwise-bounded, and satisfies
    \begin{equation}\label{eq:ptbd}
      \lim_\lambda \|\sigma_\lambda(a) - \sigma_\lambda'(a)\| = 0, \qquad a \in A;
    \end{equation}
  \item\label{ptbd3} there is an approximate morphism $(\sigma'_\lambda) \colon A \rightarrow B$, indexed by the same directed set, that is self-adjoint, pointwise-bounded, equicontinuous, and satisfies \eqref{eq:ptbd}.
  \end{enumerate}
\end{proposition}

\begin{proof}
  The bound in \ref{ptbd1} is standard: see the proof of \cite[Proposition~25.1.3]{Blackadar-K}, for example.
	
  Let $\Lambda$ be the index set of the net.  For \ref{ptbd2}, given $a\in A$, use (i) to obtain $\lambda_a \in \Lambda$ such that $\|\sigma_\lambda(a)\| \leq \|a\| + 1$ for all $\lambda \geq \lambda_a$.
  For $\lambda \in \Lambda$, define $\sigma_\lambda'' \colon A \rightarrow B$ by $\sigma_\lambda''(a) = \sigma_\lambda(a)$ for $\lambda \geq \lambda_a$ and $\sigma_\lambda''(a) = 0$ for $\lambda \not\geq \lambda_a$.
  Then $(\sigma''_\lambda)$ is a pointwise-bounded approximate morphism and hence induces a $^*$-homomorphism $A \rightarrow B_\Lambda$.
  Let $(\sigma'_\lambda)$ be a self-adjoint linear lift of this $^*$-homomorphism to $\ell^\infty(\Lambda, B)$.
	
  Condition~\ref{ptbd3} is a standard consequence of the Bartle--Graves selection theorem (\cite[Theorem~4]{Bartle-Graves52}; see also \cite[p. 85]{Dunford-Schwartz88}): every bounded linear surjective map between Banach spaces admits a continuous (not necessarily linear) splitting.
  Indeed, by \ref{ptbd2}, we may assume $(\sigma_\lambda)$ is pointwise-bounded.
  Let $\sigma_\Lambda \colon A \rightarrow B_\Lambda$ be the induced $^*$-homomorphism, and let $f' \colon B_\Lambda \rightarrow \ell^\infty(\Lambda, B)$ be a continuous splitting of $q^B_\Lambda$.
  Define $f \colon B_\Lambda \rightarrow \ell^\infty(\Lambda, B)$ by $f(b) = \frac12(f'(b) + f'(b^*)^*)$ for $b \in B$.
  Then $f$ is a continuous splitting of $q^\Lambda_B$ satisfying $f(b^*) = f(b)^*$ for all $b \in B_\Lambda$.
  For each $\lambda$, let $\sigma'_\lambda \colon A \rightarrow B$ be the $\lambda$-component of the composition $f \sigma_\Lambda$.
  Then $(\sigma'_\lambda)$ is a self-adjoint pointwise-bounded approximate morphism satisfying \eqref{eq:ptbd}.
  To prove that it is equicontinuous, fix $a_0 \in A$ and $\epsilon > 0$.
  Let $\delta > 0$ be such that for every $b \in B_\Lambda$ with $\|b - \sigma_\Lambda(a_0)\| < \delta$, we have $\|f(b) - f(\sigma_\Lambda(a_0))\| < \epsilon$.
  For $a \in A$ with $\|a - a_0\| < \delta$, we have $\|\sigma_\Lambda(a) - \sigma_\Lambda(a_0)\| < \delta$, because $^*$-homomorphisms are contractive, and so $\|f(\sigma_\Lambda(a)) - f(\sigma_\Lambda(a_0))\| < \epsilon$.
  It follows that $\|\sigma_\lambda(a) - \sigma_\lambda(a_0)\| < \epsilon$ for all $\lambda$.
\end{proof}

\begin{definition}\label{def:asymp-morph}
  For $C^*$-algebras $A$ and $B$, an \emph{asymptotic morphism} $\phi \colon A \xrightarrow\approx B$ is an approximate morphism $(\phi_t) \colon A \rightarrow B$ indexed by $\mathbb R_+$ such that the function $\mathbb R_+ \rightarrow B \colon t \rightarrow \phi_t(a)$ is continuous for all $a \in A$.
  For asymptotic morphisms $\phi, \psi \colon A \xrightarrow\approx B$, we say $\phi$ and $\psi$ are \emph{equivalent} and write $\phi \cong \psi$ if $\lim_t \|\phi_t(a) - \psi_t(a)\| = 0$ for all $a \in A$.
\end{definition}

If $A$, $B$, $D$, and $E$ are $C^*$-algebras, $\phi\colon A \xrightarrow\approx B$ is an asymptotic morphism, and $\psi \colon B \rightarrow D$ and $\theta \colon E \rightarrow A$ are $^*$-homomorphisms, then there are asymptotic morphisms $\psi \phi \colon A \xrightarrow\approx D$ and $\phi \theta \colon E \xrightarrow\approx B$ given by $(\psi\phi)_t(a) = \psi(\phi_t(a))$ and $(\phi\theta)_t(e) = \phi_t(\theta(e))$ for $t \in \mathbb R_+$, $a \in A$, and $e \in E$.

Note that all asymptotic morphisms are pointwise-bounded.
Indeed, given an asymptotic morphism $\phi \colon A \xrightarrow\approx B$ and $a \in A$, Proposition~\ref{prop:ptwise-bounded} implies there is $t_0 \in \mathbb R_+$ such that $\|\phi_t(a)\| \leq \|a\| + 1$ for all $t \in \mathbb R_{\geq t_0}$.
Since the continuous function $t \mapsto \|\phi_t(a)\|$ is necessarily bounded on the compact set $[0, t_0]$, the function $t \mapsto \|\phi_t(a)\|$ is bounded on $\mathbb R_+$.
In particular, an asymptotic morphism $\phi \colon A \xrightarrow\approx B$ induces a $^*$-homomorphism $\phi_{\rm as} \colon A \rightarrow B_{\rm as} = C_b(\mathbb R_+, B) / C_0(\mathbb R_+, B)$.
Note also that if $\phi, \psi \colon A \xrightarrow\approx B$ are asymptotic morphisms, then $\phi \cong \psi$ if and only if $\phi_{\rm as} = \psi_{\rm as}$.

We end this section with a useful procedure for constructing
asymptotic morphisms, due to Dadarlat \cite{Dadarlat94}.
We write $(\underline A, \underline \alpha)$ to denote a sequential inductive system
\begin{equation}
\begin{tikzcd}
	A_1 \arrow{r}{\alpha_1} & A_2 \arrow{r}{\alpha_2} & A_3 \arrow{r}{\alpha_3} & \cdots
\end{tikzcd}	
\end{equation}
of $C^*$-algebras.  For integers $m > n \geq 1$, let $\alpha_{n, n} =
\mathrm{id}_{A_n}$ and 
\begin{equation}
\alpha_{m, n} = \alpha_{m-1} \cdots
 \alpha_{n+1} \alpha_n.
\end{equation}  
If $A$ is the inductive limit of $(\underline A, \underline \alpha)$, let $\alpha_{\infty, n} \colon A_n \rightarrow A$ be the canonical map.

The following is a slight variation of Dadarlat's homotopy limit functor in \cite{Dadarlat94}.
The evaluation sequence $(t_n)_{n=1}^\infty$ is not included in \cite{Dadarlat94}, where the asymptotic morphisms are only considered up to homotopy, but it will be crucial for us since changing the evaluation sequence typically does change the equivalence class of the asymptotic morphism.

\begin{definition}\label{def:diag-rep}
  Let $(\underline A, \underline \alpha)$ be an inductive system of $C^*$-algebras with limit $A$ and let $B$ be a $C^*$-algebra.
  A \emph{diagrammatic representation} of an asymptotic morphism $A \xrightarrow\approx B$ is a triple $(\underline \phi, \underline h, \underline t) \colon (\underline A, \underline \alpha) \rightarrow B$ consisting of a  sequence of $^*$-ho\-mo\-morph\-isms $(\phi_n \colon A_n \rightarrow B)_{n=1}^\infty$, a sequence of homotopies $(h_n \colon A_n \rightarrow IB)_{n=1}^\infty$, and an unbounded strictly increasing sequence $(t_n)_{n=1}^\infty$ of strictly positive real numbers such that
  \begin{equation}\label{eq:diag-rep1}
    \mathrm{ev}^B_0 h_n = \phi_n \qquad \text{and} \qquad \mathrm{ev}^B_1 h_n = \phi_{n+1} \alpha_n.
  \end{equation}
  Given such a triple $(\underline \phi, \underline h, \underline t)$, define $\Phi_n \colon A_n \rightarrow C_b(\mathbb R_{\geq t_n}, B)$ by
  \begin{equation}\label{eq:diag-rep2}
    \Phi_n(a)(t) = h_m\big(\alpha_{m, n}(a)\big)\Big(\frac{t - t_m}{t_{m+1} - t_m} \Big),\quad a \in A_n,\ t_m \leq t < t_{m+1},\ m \geq n.
  \end{equation}
  Then there is a commuting diagram
  \begin{equation}\label{eq:diag-rep3}
    \begin{tikzcd}
      A_1 \arrow{r}{\alpha_1} \arrow{d}{\Phi_1} & A_2 \arrow{r}{\alpha_2} \arrow{d}{\Phi_2} & A_3 \arrow{r}{\alpha_3} \arrow{d}{\Phi_3} & \cdots &[-5ex] A\hphantom{,} \arrow[dashed]{d}{\phi_{\rm as}}	\\
      C_b(\mathbb R_{\geq t_1}, B) \arrow{r}{\rho^B_1} & C_b(\mathbb R_{\geq t_2}, B) \arrow{r}{\rho^B_2} & C_b(\mathbb R_{\geq t_3}, B) \arrow{r}{\rho^B_3} \arrow{r} & \cdots & B_{\rm as},
    \end{tikzcd}
  \end{equation}
  where the maps $\rho^B_n$ are the restriction maps.
  The fact that this diagram commutes implies that there is an induced $^*$-homomorphism $\phi_{\rm as} \colon A \rightarrow B_{\rm as}$.
  The corresponding asymptotic morphism $\phi \colon A \xrightarrow\approx B$ is the one represented by $(\underline \phi, \underline h, \underline t)$.
\end{definition}

The name ``diagrammatic representation'' was chosen since we typically picture the pair $(\underline \phi, \underline h)$ as a diagram
\begin{equation}\label{eq:diag-rep4}
  \begin{tikzcd}[row sep = 4ex, column sep = 4ex]
    A_1 \arrow{rr}{\alpha_1} \arrow{dr}{h_1} \arrow{dd}{\phi_1} & & A_2 \arrow{dd}{\phi_2} \arrow{rr}{\alpha_2} \arrow{dr}{h_2} \arrow{dd}{\phi_2} & & A_3 \arrow{rr}{\alpha_3} \arrow{dr}{h_3} \arrow{dd}{\phi_3} & & A_4 \arrow{dd}{\phi_4} \arrow[phantom]{rr}[description]{\cdots} &  & \phantom{A_5} \\
    & IB \arrow[shift left = .5ex]{dr}[near start]{\mathrm{ev}^B_1} \arrow[shift right = .5ex]{dr}[near start, swap]{\mathrm{ev}^B_0} & & IB \arrow[shift left = .5ex]{dr}[near start]{\mathrm{ev}^B_1} \arrow[shift right = .5ex]{dr}[near start, swap]{\mathrm{ev}^B_0} & & IB \arrow[shift left = .5ex]{dr}[near start]{\mathrm{ev}^B_1} \arrow[shift right = .5ex]{dr}[near start, swap]{\mathrm{ev}^B_0} & & & \\
    B \arrow[equals]{rr} & & B \arrow[equals]{rr} & & B  \arrow[equals]{rr} & & B \arrow[phantom]{rr}[description]{\cdots} & & \phantom{B}
  \end{tikzcd}
\end{equation}
with commuting triangles.
Although the choice evaluation sequence $\underline t$ is not recorded in this schematic, it is crucial in the definition of the induced asymptotic morphism.

We establish some notation for composing a diagrammatic representation of an asymptotic morphism with a $^*$-homomorphism.
Suppose $A$, $B$, and $D$ are $C^*$-algebras, $\phi \colon A \xrightarrow\approx B$ is an asymptotic morphism, and $\psi \colon B \rightarrow D$ is a $^*$-homomorphism.
Let $(\underline A, \underline \alpha)$ be an inductive limit of $C^*$-algebras with limit $A$ and let $(\underline \phi, \underline h, \underline t) \colon (\underline A, \underline \alpha) \rightarrow B$ be a diagrammatic representation of $\phi$.
Then the triple $(\underline \theta, \underline k, \underline t) \colon (\underline A, \underline \alpha) \rightarrow D$ is a diagrammatic representation of the asymptotic morphism $\psi \phi$, where $\theta_n = \psi \phi_n$ and $k_n = (I\psi) h_n$ for all $n \geq 1$.
The triple $(\underline \theta, \underline k, \underline t)$ will be denoted $\psi_*(\underline \phi, \underline h, \underline t)$.

The following lemma provides a simple sufficient condition for two diagrammatic representations to represent the same asymptotic morphism.

\begin{lemma}\label{lem:close-diagrams}
  Let $A$ and $B$ be $C^*$-algebras with $A$ separable and fix an inductive system $(\underline A, \underline \alpha)$ with limit $A$.
  Suppose that for each $n \geq 1$, a finite set $\mathcal F_n \subseteq A_n$ and $\epsilon_n > 0$ are given such that $\lim_n \epsilon_n = 0$, $\alpha_n(\mathcal F_n) \subseteq \mathcal F_{n+1}$ for all $n \geq 1$, and $\bigcup_{n=1}^\infty \alpha_{\infty, n}(\mathcal F_n)$ is dense in $A$. If
  \begin{equation}
    (\underline \phi, \underline h, \underline t), (\underline \phi', \underline h', \underline t) \colon (\underline A, \underline \alpha) \rightarrow B
  \end{equation}
  are diagrammatic representations of asymptotic morphisms $\phi, \phi' \colon A \xrightarrow\approx B$ such that
  $\|h_n(a) - h_n'(a)\| < \epsilon_n$ for all $a \in \mathcal F_n$ and $n \geq 1$, 
  then $\phi \cong \phi'$.
\end{lemma}

\begin{proof}
  Let $\Phi_n$ be given as in \eqref{eq:diag-rep2} and define $\Phi'_n$ analogously.
  Further, let $\rho^B_n$ be the restriction maps as in \eqref{eq:diag-rep3}.
  For $m \geq n \geq 1$, $a \in \mathcal F_n$, and $t \in \mathcal [t_m, t_{m+1}]$, we have $\|\Phi_m(a)(t) - \Phi_m'(a)(t)\| < \epsilon_m$ since $\alpha_{m, n}(a) \in \mathcal F_m$.
  This implies that for all $n \geq 1$,
  \begin{align} \rho^B_{\infty, n}\big(\Phi_n(a)\big) &= \rho^B_{\infty, n}\big(\Phi'_n(a)\big), \quad a \in \mathcal F_n,
  \intertext{or equivalently,} \phi_{\rm as}\big(\alpha_{\infty, n}(a)\big) &= \phi'_{\rm as} \big(\alpha_{\infty, n}(a)\big), \quad a \in \mathcal F_n.
  \end{align}
  Since $\bigcup_{n=1}^\infty \alpha_{\infty, n}(\mathcal F_n)$ is dense in $A$, it follows that $\phi_{\rm as} = \phi'_{\rm as}$, so $\phi \cong \phi'$.
\end{proof}

Note that the finite sets $\mathcal F_n$ required in Lemma~\ref{lem:close-diagrams} always exist by following simple result.

\begin{lemma}\label{lem:finite-sets}
  If $(\underline A, \underline \alpha)$ is an inductive system of separable $C^*$-algebras with limit $A$, then there are finite sets $\mathcal F_n \subseteq A_n$ such that $\alpha_n(\mathcal F_n) \subseteq \mathcal F_{n+1}$ for all $n \geq 1$ and $\bigcup_{n=1}^\infty \alpha_{\infty, n}(\mathcal F_n)$ is dense in $A$.
\end{lemma}

\begin{proof}
  Let $(\mathcal G_n)_{n=1}^\infty$ be an increasing sequence of finite subsets of $A$ with dense union.
  Given $n \geq 1$ and $a \in \mathcal G_n$, for each integer $m$ such that $1 \leq m \leq n$ and there exists $a_n \in A_n$ with $\|\alpha_{\infty, n}(a_n) - a \| < \frac1m$, choose one such $a_n$ and let $\mathcal F_n'$ denote the set of all such choices.  Then $\mathcal F_n'$ is a finite subset of $A_n$ (which may be empty).
  Define finite sets $\mathcal F_n \subseteq A_n$ by $\mathcal F_1 = \mathcal F_1'$ and $\mathcal F_{n+1} = \mathcal F_{n+1}' \cup \alpha_n(\mathcal F_n)$ for $n \geq 1$.
	
  Fix $\epsilon > 0$ and $a \in A$.  Let $m \geq 1$ be such that $\frac1m < \epsilon$.  Then choose an integer $n \geq 1$ such that there is $a' \in  \mathcal \mathcal G_n$ with $\|a - a'\| < \epsilon$.  Enlarging $n$ if necessary, we may further assume that there exists $a_n \in A_n$ such that $\|\alpha_{\infty, n}(a_n) - a'\| < \frac1m$.  By the choice of $\mathcal F_n'$, there exists $a_n' \in \mathcal F_n' \subseteq \mathcal F_n$ satisfying $\|\alpha_{\infty, n'}(a_n') - a' \| < \frac1m < \epsilon$.  Then $\|\alpha_{\infty, n}(a_n') - a \| < 2\epsilon$, and this shows that $\bigcup_{n=1}^\infty \alpha_{\infty, n}(\mathcal F_n)$ is dense in $A$.
\end{proof}

We end this section by noting that shifting a diagrammatic representation does not change the asymptotic morphism it represents.
In the notation of Definition~\ref{def:diag-rep}, for each integer $n_0 \geq 1$, let $(\underline \phi, \underline h, \underline t)_{+n_0} \colon (\underline A, \underline \alpha) \rightarrow B$ be the diagrammatic representation given by the sequences $(\phi_{n+n_0}\alpha_{n+n_0, n})_{n=1}^\infty$, $(h_{n+n_0}\alpha_{n+n_0, n})_{n=1}^\infty$, and $(t_{n+n_0})_{n=1}^\infty$.

\begin{lemma}\label{lem:shift}
  Consider an inductive system $(\underline A, \underline \alpha)$, a $C^*$-algebras $B$, a diagrammatic representation $(\underline \phi, \underline h, \underline t) \colon (\underline A, \underline \alpha) \rightarrow B$ of an asymptotic morphism $\phi \colon A \xrightarrow\approx B$.
  For any integer $n_0 \geq 1$, $(\underline \phi, \underline h, \underline t)_{+n_0}$ is a diagrammatic representation of $\phi$.
\end{lemma}

\begin{proof}
  For notational convenience, we take $n_0 = 1$.
  Note that applying this special case inductively proves the general case.
	
  For all $n$, define $\Phi_n$ as in \eqref{eq:diag-rep2} and let $\rho^B_n$ be the restriction map as in \eqref{eq:diag-rep3}.
  In parallel with \eqref{eq:diag-rep2}, define $\Phi'_n \colon A_n \rightarrow C_b(\mathbb R_{\geq t_{n +1}}, B)$ by
  \begin{equation}
    \Phi'_n(a)(t) = h_m\big(\alpha_{m, n}(a)\big)\Big(\frac{t - t_m}{t_{m+1} - t_m} \Big),\ \ a \in A_n,\ t_m \leq t < t_{m+1},\ m \geq n+1.
  \end{equation}
  Then there is a commuting diagram 
  \begin{equation}
    \begin{tikzcd}
      A_1 \arrow{r}{\alpha_1} \arrow{d}{\Phi'_1} & A_2 \arrow{r}{\alpha_2} \arrow{d}{\Phi'_2} & A_3 \arrow{r}{\alpha_3} \arrow{d}{\Phi'_3} & \cdots &[-5ex] A\hphantom{,} \arrow[dashed]{d}{\phi'_{\rm as}}	\\
      C_b(\mathbb R_{\geq t_2}, B) \arrow{r}{\rho^B_2} & C_b(\mathbb R_{\geq t_3}, B) \arrow{r}{\rho^B_3} & C_b(\mathbb R_{\geq t_4}, B) \arrow{r}{\rho^B_4} \arrow{r} & \cdots & B_{\rm as},
    \end{tikzcd}
  \end{equation}
  defining an asymptotic morphism $\phi' \colon A \xrightarrow\approx B$ corresponding to $(\underline \phi, \underline h, \underline t)_{+1}$.
  We have $\Phi'_n = \rho^B_n \Phi_n$ for all $n \geq 1$, and hence
  \begin{equation}
    \phi'_{\rm as} \alpha_{\infty, n} = \rho^B_{\infty, n+1} \Phi'_n = \rho^B_{\infty, n} \Phi_n = \phi_{\rm as} \alpha_{\infty, n}
  \end{equation}
  for all $n \geq 1$.  Therefore, $\phi'_{\rm as} = \phi_{\rm as}$, and hence $\phi' \cong \phi$.
\end{proof}

\section{Approximate decomposability}
\label{sec:qdl}

An extension $0 \rightarrow \ker(\pi) \rightarrow E \xrightarrow \pi B \rightarrow 0$ of $C^*$-algebras is called \emph{quasidiagonal} if $\ker(\pi)$ admits a approximate unit of projections $(p_\lambda)$ that is quasicentral in $E$, in the sense that $\|p_\lambda e - e p_\lambda\| \rightarrow 0$ for all $e \in E$.
Quasidiagonal extensions were introduced in \cite{Brown84} in connection with extension theory.
It is typically required that the net $(p_\lambda)$ is an increasing sequence of projections, but it is not hard to show that this coincides with the present definition when the $C^*$-algebras involved are separable.

For our purposes, the most important property of a quasidiagonal extension is the existence of an approximate direct sum decomposition of such an extension.
For a quasidiagonal extension $0 \rightarrow \ker(\pi) \rightarrow E \xrightarrow \pi B \rightarrow 0$, if $\sigma \colon B \rightarrow E$ is any (set-theoretic) splitting, then the maps
\begin{equation}
  \sigma_\lambda \colon B \rightarrow E\colon b \mapsto  (1 - p_\lambda) \sigma(b)
\end{equation}
form an approximate morphism, as do the maps
\begin{equation}
  \rho_\lambda \colon E \rightarrow \ker(\pi) \colon e \mapsto
  p_\lambda e.
\end{equation}
Further, the approximate morphisms
\begin{align}
  \ker(\pi) \oplus B \rightarrow E&\colon (x, b) \mapsto x + \sigma_\lambda(b)
                                    \intertext{and}
                                    E \rightarrow \ker(\pi) \oplus B &\colon e \mapsto \big(\rho_\lambda(e), \pi(e)\big)
\end{align}
are approximately inverses of each other.
In this sense, there is an approximate decomposition of $E$ into a direct sum $\ker(\pi) \oplus B$

We introduce the following weaker condition on an extension that still allows for the approximate direct sum decomposition (Theorem~\ref{prop:approx-direct-sum}).
Roughly, we require an approximate splitting of the quotient map that is both an approximate morphism and an approximate bimodule map over the extension algebra.

\begin{definition}
  \label{def:qdelicious}
  An extension $0 \rightarrow \ker(\pi) \rightarrow E \xrightarrow \pi B \rightarrow 0$ is \emph{approximately decomposable}, or more briefly, $\pi$ is \emph{approximately decomposable}, if there is an approximate morphism $(\sigma_\lambda)\colon B\to E$ such that
  \begin{enumerate}
  \item\label{qdl-split} $\lim_\lambda \|\pi(\sigma_\lambda(b)) - b\| = 0$ and
  \item\label{qdl-bimod} $\lim_\lambda \| \sigma_\lambda(\pi(e)b) - e\sigma_\lambda(b) \|
    = 0$
  \end{enumerate}
  for all $b\in B$ and $e\in E$.
\end{definition}

The following gives a connection between quasidiagonality and approximate decomposability.
The condition in the second sentence of the theorem is essentially the definition of weak quasidiagonality in \cite[Definition~2.1]{Manuilov-Thomsen00}, which is defined for separable extensions with stable ideal and requires that the net of projections be an increasing sequence.
In particular, all weakly quasidiagonal extensions are approximately decomposable.

\begin{proposition}\label{prop:qd-qdl}
  Every quasidiagonal extension  is approximately decomposable.  More generally, if $\pi \colon E \rightarrow B$ is a surjective $^*$-homomorphism and there is a net of projections $(p_\lambda)$ in $M(E)$, the multiplier algebra of $E$, such that
  \begin{enumerate}
  \item\label{wqd1}  $p_\lambda E \subseteq \ker(\pi)$,
  \item\label{wqd2}  $\lim_\lambda \|p_\lambda e - e p_\lambda\| = 0$ for all $e \in E$, and
  \item\label{wqd3} $\lim_\lambda \|p_\lambda x - x \| = 0$ for all $x \in \ker(\pi)$,
  \end{enumerate}
  then $\pi$ is approximately decomposable.
\end{proposition}

\begin{proof}
  It suffices to prove the second sentence.
  Let $\sigma \colon B \rightarrow E$ be a self-adjoint linear map with $\pi \sigma = \mathrm{id}_B$ and define $\sigma_\lambda \colon B \rightarrow E$ by $\sigma_\lambda(b) = (1 - p_\lambda) \sigma(b) $.
  For $b_1, b_2 \in B$, \ref{wqd2} implies
  \begin{equation}\label{eq:wqd1}
    \lim_\lambda \|\sigma_\lambda(b_1) \sigma_\lambda(b_2) - (1 - p_\lambda) \sigma(b_1) \sigma(b_2) \| = 0.
  \end{equation}
  Since $\pi \sigma = \mathrm{id}_B$ and $\pi$ is a $^*$-homomorphism, we have $\sigma(ab) - \sigma(a) \sigma(b) \in \ker(\pi)$.
  Therefore, \ref{wqd3} yields
  \begin{equation}\label{eq:wqd2}
    \lim_\lambda \|  (1 - p_\lambda) \sigma(b_1)\sigma(b_2) -  \sigma_\lambda(b_1b_2)\| = 0.
  \end{equation}
  Combining \eqref{eq:wqd1} and \eqref{eq:wqd2} shows $(\sigma_\lambda) \colon B \rightarrow E$ is approximately multiplicative.  Further, as $\sigma$ and $p_\lambda$ are self-adjoint, \ref{wqd2} implies
  \begin{equation}
  	\lim_{\lambda} \|\sigma_\lambda(b^*) - \sigma_\lambda(b)^*\| = 0, \quad b \in B.
  \end{equation}
  Since each $\sigma_\lambda$ is linear, we have that $(\sigma_\lambda) \colon B \rightarrow E$ is an approximate morphism.  For $b \in B$ and $e \in E$, $e \sigma(b) - \sigma(\pi(e)b) \in \ker(\pi)$, so using \ref{wqd2} and \ref{wqd3} as above, we have Definition~\ref{def:qdelicious}\ref{qdl-bimod} holds.
  Moreover, \ref{wqd1} implies $p_\lambda \sigma(b) \in \ker(\pi)$, and hence $\pi(\sigma_\lambda(b)) = b$ for all $b \in B$ and $\lambda$.  In particular, Definition~\ref{def:qdelicious}\ref{qdl-split} holds.
\end{proof}

The converse of the previous result holds in the unital case.

\begin{proposition}
  Every unital approximately decomposable extension is quasidiagonal.
\end{proposition}

\begin{proof}
  Let $0 \rightarrow \ker(\pi) \rightarrow E \xrightarrow \pi B \rightarrow 0$ be an approximately decomposable extension such that $E$ is unital and fix a approximate morphism $(\sigma_\lambda) \colon B \rightarrow E$ as in Definition~\ref{def:qdelicious}.
  By Definition~\ref{def:qdelicious}\ref{qdl-split}, we have
  \begin{equation}\label{eq:qd1}
    \lim_\lambda \big\| \pi\big( 1_E - \sigma_\lambda(1_B) \big) \big\| = 0,
  \end{equation}
  and hence there is a net $(p'_\lambda) \subseteq \ker(\pi)$ such that
  \begin{equation}\label{eq:qd2}
    \lim_\lambda \|1_E - \sigma_\lambda(1_B) - p_\lambda'\| = 0.
  \end{equation}
  Since $\sigma_\lambda$ is a approximate morphism, we obtain $\lim_\lambda\|p_\lambda' - (p_\lambda')^*p_\lambda' \| = 0$.
  Therefore, there is a net of projections $(p_\lambda) \in \ker(\pi)$ such that $\lim_\lambda \|p_\lambda - p_\lambda'\| = 0$.
  Then \eqref{eq:qd2} implies
  \begin{equation}\label{eq:qd3}
    \lim_\lambda \|1_E - \sigma_\lambda(1_B) - p_\lambda\| = 0.
  \end{equation}
  We will show that $(p_\lambda)$ is an approximate unit for $\ker(\pi)$ that is quasicentral in $E$.
	
  Let $x \in \ker(\pi)$ and $e \in E$ be given.
  Using the bimodule property of Definition~\ref{def:qdelicious}\ref{qdl-bimod},
  \begin{equation}\label{eq:qd4}
    x\big( 1_E -\sigma_\lambda(1_B) \big) = x - x \sigma_\lambda(1_B) \approx x - \sigma_\lambda(\pi(x) 1_B) = x,
  \end{equation}
  where ``$\approx$'' means equality in the limit over $\lambda$.
  It follows from \eqref{eq:qd3} and \eqref{eq:qd4} that $(p_\lambda)$ is an approximate unit for $\ker(\pi)$.
  Similarly, Definition~\ref{def:qdelicious}\ref{qdl-bimod} implies
  \begin{align}
  	e \big(1_E - \sigma_\lambda(1_B)\big) &= e - e \sigma_\lambda(1_B) \approx e - \sigma_\lambda\big(\pi(e)\big),
  	\intertext{and replacing $e$ and $x$ with $e^*$ and $x^*$ and taking adjoints shows}
    \big( 1_E - \sigma_\lambda(1_B) \big) e &= e - \sigma_\lambda(1_B) e \approx e - \sigma_\lambda\big(\pi(e)\big).         
  \end{align}
  These two approximations combine with \eqref{eq:qd3} to show $\lim_\lambda \|ep_\lambda - p_\lambda e\| = 0$.
\end{proof}

The following lemma shows that one can perturb the approximate morphism in Definition~\ref{def:qdelicious} to be as in Proposition~\ref{prop:ptwise-bounded}\ref{ptbd2} while also satisfying Definition~\ref{def:qdelicious}\ref{qdl-split} exactly instead of approximately.
This will be further strengthened in Proposition~\ref{prop:approx-direct-sum}.

\begin{lemma}
  \label{lem:qdelicious-bounded}
  Let $\pi \colon E\to B$ be a surjective approximately decomposable $^*$-ho\-mo\-mor\-phism between $C^*$-algebras and suppose $(\sigma_\lambda)\colon B\to E$ is as in Definition~\ref{def:qdelicious}.
  Then there is an approximate morphism $(\sigma_\lambda')\colon B\to E$, indexed by the same directed set, that is self-adjoint, linear, pointwise-bounded, satisfies $\lim_\lambda \| \sigma_\lambda'(b) - \sigma_\lambda(b)\| = 0$ for all $b\in B$, and is such that $\pi \sigma_\lambda' = \id_B$ for all $\lambda\in \Lambda$.
\end{lemma}

\begin{proof}
  We may assume that $(\sigma_\lambda)_{\lambda\in \Lambda}$ is self-adjoint, linear, and pointwise-bounded by Proposition~\ref{prop:ptwise-bounded}.
  Define $P$ to be the pullback of $E_\Lambda$ and $\ell^\infty(\Lambda, B)$ along the map $\pi_\Lambda \colon E_\Lambda\to B_\Lambda$ induced by $\pi$ and the canonical surjection $q_\Lambda^B \colon \ell^\infty(\Lambda, B) \to B_\Lambda$.
  The rest of the proof is essentially contained in the diagram
  \begin{equation}
    \label{eq:qdelicious-bounded-pullback}
    \begin{tikzcd}
      B
      \arrow[bend right = 25]{dddr} [swap]{\sigma_\Lambda}
      \arrow[bend left = 15]{drrr} {\tilde\iota^B_\Lambda}
      \arrow[dashed]{dr}{\tilde\sigma_\Lambda}
      & & & &\\
      & P
      \ar{rr}
      \ar{dd}
      & & \ell^\infty(\Lambda, B)
      \arrow{dd}{q^B_\Lambda}\\
      & & \ell^\infty(\Lambda, E)
      \arrow[from=uull, dashed, crossing over, bend right]{}[swap]{\rho}
      \arrow[dashed]{ul}
      \arrow{ur}[swap]{\Lambda\pi}
      \arrow{dl}[yshift = 4pt]{q^E_\Lambda}
      & \\
      & E_\Lambda
      \arrow{rr}[swap]{\pi_\Lambda}
      & & B_\Lambda
    \end{tikzcd}
  \end{equation}
  as we will explain.
  At the moment, the solid arrows in \eqref{eq:qdelicious-bounded-pullback} are defined and the diagram formed by removing the dashed arrows and $B$ commutes.
  
  The approximate morphism $(\sigma_\lambda)$ induces a $^*$-homomorphism $\sigma_\Lambda\colon B\to E_\Lambda$.
  Let $\tilde\iota^B_\Lambda\colon B\to \ell^\infty(\Lambda, B)$ be the canonical embedding.
  Then $\pi_\Lambda \sigma_\Lambda = q^B_\Lambda \tilde\iota^B_\Lambda$, and so, from the universal property of $P$, we obtain a $^*$-homomorphism $\tilde\sigma_\Lambda\colon B\to P$ making each of the triangles with vertices $B$, $P$, $E_\Lambda$ and $B$, $P$, $\ell^\infty(\Lambda, B)$ commute.
  A similar argument gives the existence of a $^*$-homomorphism $\ell^\infty(\Lambda, E)\to P$ making the triangles with vertices $\ell^\infty(\Lambda, E)$, $P$, $E_\Lambda$ and $\ell^\infty(\Lambda, E)$, $P$, $\ell^\infty(\Lambda, B)$ commute. Note that this map is surjective by the surjectivity of $q^E_\Lambda$.
	
  Let $\rho\colon B\to \ell^\infty(\Lambda, E)$ be a self-adjoint linear lift of $\tilde\sigma_\Lambda$ along the unlabeled dashed arrow.
  Commutativity of \eqref{eq:qdelicious-bounded-pullback} shows that $\sigma_\Lambda = q^E_\Lambda \rho$.
  If $(\sigma_\lambda')$ is the approximate morphism determined by the coordinates of $\rho$, then $(\sigma_\lambda')$ is a self-adjoint linear pointwise-bounded approximate morphism with $\lim_\lambda \|\sigma'_\lambda(b) - \sigma_\lambda(b) \| = 0$ for all $b \in B$ (since $\sigma_\Lambda = q^E_\Lambda \rho$).
  Further, since $(\Lambda \pi)\rho = \tilde\iota^B_\Lambda$, we have $\pi\sigma'_\lambda = \mathrm{id}_B$ for all $\lambda$.
\end{proof}

The following result gives an approximate direct sum decomposition of approximately decomposable extensions, which is the reason for their name.
This is usually the most relevant property of quasidiagonal extensions and why we believe approximately decomposable extensions provide a useful generalization.
The converse also holds---see Remark~\ref{rem:approx-direct-sum}.

\begin{proposition}
  \label{prop:approx-direct-sum}
  Let $0\rightarrow J \xrightarrow{\iota} E \xrightarrow \pi B \rightarrow 0$ be an approximately decomposable extension of $C^*$-algebras.
  Then there are self-adjoint linear pointwise-bounded approximate morphisms $(\rho_\lambda)\colon E\to J$ and $(\sigma_\lambda)\colon B\to E$, indexed by the same directed set, such that for every $\lambda$,
  \begin{equation}\label{eq:approx-direct-sum}
    \rho_\lambda \iota = \id_{J}, \quad
    \pi\sigma_\lambda = \id_B, \quad\text{and} \quad
    \iota \rho_\lambda + \sigma_\lambda \pi = \id_E.
  \end{equation}
\end{proposition}

\begin{proof}
  Using Lemma~\ref{lem:qdelicious-bounded}, fix a self-adjoint linear pointwise-bounded approximate morphism $(\sigma_\lambda)\colon B\to E$ such that $\pi\sigma_\lambda = \id_B$ for all $\lambda$ and
  \begin{equation}
    \label{eq:approx-sum-cond}
    \lim_\lambda \big\| \sigma_\lambda\big(\pi(e)b\big) - e\sigma_\lambda(b) \big\| = 0, \qquad b \in B,\ e \in E.
  \end{equation}
  Notice that $\pi  (\id_E - \sigma_\lambda \pi)=0$ for all $\lambda$, so for all $\lambda$ there is a self-adjoint linear map $\rho_\lambda\colon E\to J$ such that $\iota\rho_\lambda = \id_E - \sigma_\lambda \pi$.
  The fact that $\iota$ is isometric implies
  \begin{equation}
    \|\rho_\lambda(e)\| \leq
    \|e\| + \|\sigma_\lambda(\pi(e))\|
  \end{equation}
  for all $e \in E$ and all $\lambda$. Therefore, $(\rho_\lambda)$ is pointwise-bounded.
	
  It remains to show $(\rho_\lambda)$ is an approximate morphism.
  Let $e_1, e_2 \in E$.
  Using ``$\approx$'' to mean equality in the limit over $\lambda$, we compute
  \begin{align}
    \begin{split}
      \iota\big( \rho_\lambda(e_1)  \rho_\lambda(e_2) \big) &= (\mathrm{id}_E - \sigma_\lambda \pi)(e_1)(\mathrm{id}_E - \sigma_\lambda\pi)(e_2) \\
      &= e_1 e_2 - e_1 \sigma_\lambda\big(\pi(e_2)\big) - \sigma_\lambda\big(\pi(e_1)\big) e_2 + \sigma_\lambda\big(\pi(e_1)\big) \sigma_\lambda\big(\pi(e_2)\big) \\
       &\approx e_1 e_2 - \sigma_\lambda\big(\pi(e_1) \pi(e_2)\big) \\
     &= (\mathrm{id}_E - \sigma_\lambda \pi)(e_1e_2) \\
        &= \iota\big( \rho_\lambda(e_1 e_2) \big),
    \end{split}
  \end{align}
  where the approximation on the third line follows from the approximate bimodule property of the $\sigma_\lambda$ applied to the second and third terms and the approximate multiplicative property applied to the final term.
  Since $\iota$ is isometric, this ends the proof.
\end{proof}

\begin{remark}\label{rem:approx-direct-sum}
  Any extension satisfying the conclusion of Proposition~\ref{prop:approx-direct-sum} is approximately decomposable, and, moreover, the approximate morphism $(\sigma_\lambda) \colon B \rightarrow E$ necessarily satisfies the conditions of Definition~\ref{def:qdelicious}.
  Indeed, \ref{qdl-split} clearly holds.
  To see \ref{qdl-bimod}, first note that the relations in \eqref{eq:approx-direct-sum} imply $\rho_\lambda \sigma_\lambda = 0$.
  For $x \in \ker(\pi)$ and $b \in B$, we have $x \sigma_\lambda(b) \in \ker(\pi)$, and hence $\rho_\lambda(x \sigma_\lambda(b)) = x \sigma_\lambda(b)$.
  Since $(\rho_\lambda)$ is an approximate morphism and $\rho_\lambda(\sigma_\lambda(b)) =0$ for all $\lambda$, we have $\lim_\lambda \|x\sigma_\lambda(b)\| = 0$.
  Now for $b \in B$ and $e \in E$, $x = e - \sigma_\lambda(\pi(e)) \in \ker(\pi)$, and hence $ \lim_\lambda \|e \sigma_\lambda(b) - \sigma_\lambda(\pi(e)) \sigma_\lambda(b)\| = 0$.
  Finally, \ref{qdl-bimod} follows from the approximate multiplicativity of $(\sigma_\lambda)$.
\end{remark}

\section{Semiprojectivity and shape theory}\label{sec:shape}

This section provides some variations of known results regarding semiprojectivity based mostly on results of Blackadar \cite{Blackadar85, Blackadar16}.
Some of these are known to experts, but lack an explicit statement in the literature.  Others are proved for semiprojective $C^*$-algebras, but the analogous versions for semiprojective morphisms are needed here instead.
We will use these results to prove our relative homotopy lifting result, Theorem~\ref{thm:hlp-qdl}.

A $^*$-homomorphism $\alpha \colon A_0 \rightarrow A$ between separable $C^*$-algebras is \emph{semiprojective} if for every inductive system $(\underline B, \underline \beta)$ with limit $B$ such that each $\beta_n$ is surjective and for every $^*$-homomorphism $\phi \colon A \rightarrow B$, there are an integer $n \geq 1$ and a $^*$-homomorphism $\tilde \phi \colon A \rightarrow B_n$ such that $\phi \alpha = \beta_{\infty, n} \tilde \phi$.
A separable $C^*$-algebra is \emph{semiprojective} if $\mathrm{id}_A$ is semiprojective.

The definition of semiprojectivity given above is due to Blackadar in his development of shape theory for $C^*$-algebras \cite{Blackadar85}, building on early work of Effros and Kaminker \cite{Effros-Kaminker86}.
Following \cite{Blackadar85}, a \emph{shape system} for a separable $C^*$-algebra $A$ is an inductive system $(\underline A, \underline \alpha)$ of separable $C^*$-algebras with limit $A$ such that each $\alpha_n$ is semiprojective.
Every separable $C^*$-algebra admits a shape system by \cite[Theorem~4.3]{Blackadar85}.

The following result of Dadarlat will be used to construct diagrammatic representations of asymptotic morphisms.

\begin{proposition}[{\cite[Proposition 3.14]{Dadarlat94}}]\label{prop:shape-lift}
  If $A$ is a separable $C^*$-algebra, $B$ is a $C^*$-algebra, $\phi\colon A\xrightarrow\approx B$ is an asymptotic morphism , and $(\underline{A}, \underline{\phi})$ is a shape system for $A$, then there are a strictly increasing sequence $(m_n)_{n=1}^\infty$ of positive integers and a morphism
  \begin{equation}\label{eq:shape-lift}
    \begin{tikzcd}
      A_1 \arrow{r}{\alpha_1} \arrow{d}{\Phi_1} &
      A_2 \arrow{r}{\alpha_2} \arrow{d}{\Phi_2} &
      A_3 \arrow{r}{\alpha_3} \arrow{d}{\Phi_3} & \cdots\\
      C_b(\mathbb{R}_{\geq m_1}, B) \arrow{r}{\rho_1^B} &
      C_b(\mathbb{R}_{\geq m_2}, B) \arrow{r}{\rho_2^B} &
      C_b(\mathbb{R}_{\geq m_3}, B) \arrow{r}{\rho_3^B} & \cdots
    \end{tikzcd}
  \end{equation}
  of inductive systems that induces $\phi_{\rm as}$, where the $\rho^B_n$ are the restriction maps.
\end{proposition}

The following result gives the existence of diagrammatic representations of asymptotic morphisms with respect to any shape system for the domain.
The result will not be used explicitly, but the reader may find the proof instructive.
A much more refined version of this proof will appear in Lemma~\ref{lem:somehow}.

\begin{corollary}
  Suppose $A$ and $B$ are $C^*$-algebras with $A$ separable and $\phi \colon A \xrightarrow\approx B$ is an asymptotic morphism.
  If $(\underline A, \underline \alpha)$ is a shape system for $A$, then there is a diagrammatic representation $(\underline \phi, \underline h, \underline t) \colon (\underline A, \underline \alpha) \rightarrow B$ of $\phi$.
\end{corollary}

\begin{proof}
  Apply Proposition~\ref{prop:shape-lift} to obtain a commuting diagram as in \eqref{eq:shape-lift}.
  Let $t_n = m_n$ for $n \geq 1$.
  Further, for $n \geq 1$, define $\phi_n \colon A_n \rightarrow B$ by $\phi_n(a) = \Phi_n(a)(m_n)$ and $h_n \colon A_n \rightarrow IB$ by
  \begin{equation}
    h_n(a)(s) = \Phi_n(a)\big((1-s)m_n + sm_{n+1}\big), \quad a \in A_n,\ s \in I.
  \end{equation}
  Then $(\underline \phi, \underline h, \underline t)$ is a diagrammatic representation of an asymptotic morphism.
  Further, the diagram in \eqref{eq:diag-rep3} induced by $(\underline \phi, \underline h, \underline t)$ is precisely the diagram in \eqref{eq:shape-lift}.
  Therefore, $(\underline \phi, \underline h, \underline t)$ induces $\phi$.
\end{proof}

The following stability result for liftable $^*$-homomorphisms is essentially due to Blackadar.
The case when $A_0 = A$ and $\alpha = \mathrm{id}_A$ is essentially \cite[Theorem~4.1]{Blackadar16}.
The more general version here is taken from \cite[Theorem~1.7]{top-e}.

\begin{theorem}[Blackadar]\label{thm:liftable-closed}
  Let $\alpha \colon A_0 \rightarrow A$ be a semiprojective $^*$-homomorphism between separable $C^*$-algebras $A_0$ and $A$.
  For every finite set $\mathcal F\subseteq A_0$ and $\epsilon > 0$,
  there are a finite set $\mathcal G \subseteq A$ and $\delta > 0$ such that if $B$ and $E$ are $C^*$-algebras, $\pi \colon E \rightarrow B$ is a surjective $^*$-homomorphism, $\phi, \psi \colon A \rightarrow B$ are $^*$-homomorphisms with $\|\phi(a) - \psi(a)\| < \delta$ for all $a \in \mathcal G$, and $\tilde\phi \colon A \rightarrow E$ is a $^*$-homomorphism with $\pi \tilde\phi = \phi$, then there is a $^*$-homomorphism $\tilde \psi \colon A_0 \rightarrow E$ such that $\pi \tilde\psi = \psi\alpha$ and $\|\tilde\phi(\alpha(a)) - \tilde\psi(a)\| < \epsilon$ for all $a \in \mathcal F$.
\end{theorem}

The above theorem can be visualized as a diagram
\begin{equation}
  \begin{tikzcd}
    &[20pt] &[50pt] E\hphantom{,} \arrow{d}{\pi} \\[10pt]
    A_0 \arrow{r}[swap]{\alpha}[name=LD, near end]{} \arrow[bend left, dashed]{urr}{\tilde\psi}[name=LU, below]{} \arrow[to path=(LU) -- (LD)\tikztonodes, phantom]{r}[pos = .6, description]{\underset{\scriptscriptstyle \mathcal F, \epsilon}{\approx}} & A \arrow[bend left=20]{ur}[swap, pos = .7]{\tilde\phi} \arrow[bend left = 20]{r}{\psi}[name = RU,below]{} \arrow[bend right=20]{r}[swap]{\phi}[name = RD]{} 
    \arrow[to path=(RU) -- (RD)\tikztonodes, phantom]{r}[pos = .55, description]{\underset{\scriptscriptstyle \mathcal G, \delta}{\approx}} & B,
  \end{tikzcd}	
\end{equation}
taking care to note that the triangle formed by $\tilde\phi$, $\pi$, and $\psi$ does not commute.

Theorem~\ref{thm:liftable-closed} has a striking consequence: two $^*$-homomorphisms from a semiprojective $C^*$-algebra to a common codomain are homotopic whenever they are close in the point-norm topology, with the bound depending only on the domain.
Moreover, one can arrange the homotopy to be approximately constant. This is the content of \cite[Corollary~4.2]{Blackadar16}.

The following is slight strengthening of this result that provides a relative version for semiprojective $^*$-homomorphisms and allows one to control the behavior of the homotopy in a quotient.
A weaker version of this result without control on the length of the homotopy or the behavior of the homotopy in quotient appeared as \cite[Corollary~1.8]{top-e}.
	
\begin{corollary}
  \label{cor:wishlist}
  Let $\alpha\colon A_0\to A$ be a semiprojective $^*$-homomorphism between separable $C^*$-algebras $A_0$ and $A$.
  For every finite set $\mathcal{F} \subset A_0$ and $\epsilon > 0$, there are a finite subset $\mathcal{G} \subset A$ and $\delta > 0$ such that for all $C^*$-algebras $B$ and $E$ and  $^*$-ho\-mo\-mor\-phisms $\pi\colon E\to B$ and $\phi, \psi\colon A\to E$ such that $\pi$ is surjective, $\pi \phi = \pi \psi$, and $\| \phi(a) - \psi(a) \| < \delta$ for all $a \in \mathcal G$, there is a $^*$-homomorphism $h\colon A_0\to IE$ satisfying the following properties:
  \begin{enumerate}
  \item $\mathrm{ev}_0^E h = \phi\alpha$ and $\mathrm{ev}_1^E h = \psi\alpha$;
  \item $\pi( h(a)(t)) = \pi( \phi(\alpha(a)))$ for all $a\in A_0$ and $t\in I$;
  \item  $\| h(a)(t) - \phi(\alpha(a)) \| <
    \epsilon$ for all $a \in \mathcal F$ and $t\in I$.
  \end{enumerate}
\end{corollary}

\begin{proof}
  Let $\mathcal{F} \subset A_0$ and $\epsilon>0$ be given.
  Let $\mathcal{G}$ and $\delta$ be as provided by Theorem~\ref{thm:liftable-closed}.
  Given a surjective $^*$-homomorphism $\pi\colon E\to B$ and $^*$-homomorphisms $\phi, \psi\colon A\to E$ with $\pi \phi = \pi \psi$ and $\|\phi(a) - \psi(a)\| < \delta$ for all $a \in \mathcal G$, define
  \begin{align}
    \begin{split}
      \hat{E} &= \big\{ f\in IE : \pi\big(f(t)\big) =
      \pi\big(f(0) \big) \text{ for all } t\in I \big\},\\
      \hat{B} &= \big\{ (e_0, e_1) \in E\oplus E : \pi(e_0) = \pi(e_1)
      \big\},
    \end{split}
  \end{align}
  and a surjective $^*$-homomorphism
  \begin{equation}
    \hat{\pi}\colon \hat{E}\to \hat{B} \colon f\mapsto \big( f(0), f(1)
    \big).
  \end{equation}
  Then $\phi\oplus \phi$ and $\phi\oplus \psi$ are $^*$-homomorphisms $A \rightarrow \hat B$ with 
  \begin{equation}
  	\|(\phi\oplus \phi)(a) - (\phi\oplus \psi)(a)\| < \delta, \quad a\in \mathcal{G},
  \end{equation} 
  and $\phi\oplus \phi$ lifts along $\hat{\pi}$ to the $^*$-homomorphism $A\to \hat{E}$ defined by sending $a\in A$ to the constant function with value $\phi(a)$.
  Let $h\colon A_0\to \hat{E} \subseteq IE$ be the lift of $\phi\oplus \psi$ given by Theorem~\ref{thm:liftable-closed}.
\end{proof}

The next result gives a strengthening of the definition semiprojectivity by demanding that the partial lift $\tilde\phi$ in the definition extend a previously defined partial lift.
The proof below is extracted from the proof of \cite[Theorem~1.7]{top-e} (restated as Theorem~\ref{thm:liftable-closed} above) which, in turn, is modeled on \cite[Theorem~3.1]{Blackadar16}

\begin{proposition}\label{prop:lift-prescribed-quotient}
  Let $A_0$ and $A$ be separable $C^*$-algebras and let $\alpha \colon A_0 \rightarrow A$ be a semiprojective $^*$-homomorphism.
  Suppose that
  \begin{equation}
    \begin{tikzcd}
      0 \arrow{r} & J_n \arrow{r}{\iota_n} \arrow{d}{\rho_n^J} & E_n \arrow{r}{\pi_n} \arrow{d}{\rho_n^E} & B_n \arrow{r} \arrow{d}{\rho_n^B} & 0 \\
      0 \arrow{r} & J_{n+1} \arrow{r}{\iota_{n+1}} & E_{n+1} \arrow{r}{\pi_{n+1}} & B_{n+1} \arrow{r} & 0
    \end{tikzcd}
  \end{equation}
  is an inductive system of extensions of $C^*$-algebras with limit
  \begin{equation}
    \begin{tikzcd}
      0 \arrow{r} & J \arrow{r}{\iota} & E \arrow{r}{\pi} & B \arrow{r} & 0,
    \end{tikzcd}
  \end{equation}
  and assume each $\rho_n^J$ is surjective.
  Given $m \geq 1$ and $^*$-homomorphisms $\phi_m \colon A \rightarrow B_m$ and $\tilde\phi \colon A \rightarrow E$ such that $\pi \tilde\phi = \rho_{\infty, m}^B \phi_m$, there are an integer $n \geq m$ and a $^*$-homomorphism $\tilde\phi_n \colon A_0 \rightarrow E_n$ such that $\rho_{\infty, n}^E \tilde\phi_n = \tilde\phi \alpha$ and $\pi_n \tilde\phi_n = \rho_{n,m}^B \phi_m \alpha$.
\end{proposition}

\begin{proof}
  For $n \geq m$, consider the pullback $C^*$-algebra
  \begin{equation}
    P_n = E_n \oplus_{B_n} B_m = \{(e_n, b_m) \in E_n \oplus B_m : \pi_n(e_n) = \rho_{n, m}^B(b_m) \},
  \end{equation}
  as illustrated in the top-level of the diagram	
  \begin{equation}\label{eq:pullback-seq}
    \begin{tikzcd}[column sep = 2.3ex]
      & 0 \arrow{rr} & & J_n \arrow{rr}{\tilde\iota_n} \arrow{dd}[near end, swap]{\rho_n^J} \arrow[equals]{dl} & & P_n \arrow{dd}[near end, swap]{\rho_n^P} \arrow{rr}{\mathrm{pr}_n^{(2)}} \arrow{dl}[swap, yshift = -.7ex]{\mathrm{pr}_n^{(1)}} & & B_m \arrow[equals]{dd} \arrow{rr} \arrow{dl}[swap, near end]{\rho_{n, m}^B} & & 0\\
      0 \arrow{rr} & & J_n \arrow[crossing over]{rr}[near end]{\iota_n} & & E_n \arrow[crossing over]{rr}[near end]{\pi_n} & & B_n \arrow[crossing over]{rr} & & 0\hphantom{,} & \\
      & 0 \arrow{rr} & & J_{n+1} \arrow{rr}[near start]{\tilde\iota_{n+1}} \arrow[equals]{dl} & & P_{n+1} \arrow{rr}[near start]{\mathrm{pr}_{n+1}^{(2)}} \arrow{dl}[near start, xshift = -.2ex, yshift = .5ex]{\mathrm{pr}_{n+1}^{(1)}} & & B_m \arrow{rr} \arrow{dl}[near start]{\rho_{n+1, m}^B} & & 0 \\
      0 \arrow{rr} & & J_{n+1} \arrow{rr}[swap]{\iota_{n+1}} \arrow[crossing over, leftarrow]{uu}[near start]{\rho_n^J} & & E_{n+1} \arrow{rr}[swap]{\pi_{n+1}} \arrow[crossing over, leftarrow]{uu}[near start]{\rho_n^E} & & B_{n+1} \arrow{rr} \arrow[crossing over, leftarrow]{uu}[near start]{\rho_n^B} & & 0, & \\
    \end{tikzcd}
  \end{equation}
  where $\tilde{\iota}_n$ maps $x\in J_n$ to $(x, 0)\in P_n$, $\mathrm{pr}_n^{(1)}$ and  $\mathrm{pr}_n^{(2)}$ are the projections onto the first and second coordinates, respectively, and $\rho_n^P\colon P_n\to P_{n+1}$ applies $\rho_n^E$ to the first coordinate.
  It is straightforward that \eqref{eq:pullback-seq} commutes.
  

  The fact that every $\rho_n^J$ is surjective implies that every $\rho_n^P$ is surjective.
  Let $P$ be the inductive limit of $(\underline P, \underline \rho^P)$.
  Taking the inductive limit over $n$ in \eqref{eq:pullback-seq} produces a commutative diagram
  \begin{equation}\label{eq:pullback-limit}
    \begin{tikzcd}
      0 \arrow{r} & J \arrow{r}{\tilde\iota} \arrow[equals]{d} & P \arrow{r}{\mathrm{pr}_\infty^{(2)}} \arrow{d}{\mathrm{pr}_\infty^{(1)}} & B_m \arrow{r} \arrow{d}{\rho_{\infty, m}^B} & 0 \\
      0 \arrow{r} & J \arrow{r}{\iota} & E \arrow{r}{\pi} & B \arrow{r} & 0
    \end{tikzcd}
  \end{equation}
  with exact rows.
  Therefore, the right square in \eqref{eq:pullback-limit} is a pullback square.
  It follows that $\tilde\phi$ and $\phi_m$ induce a $^*$-homomorphism $\psi \colon A \rightarrow P$ with $\mathrm{pr}_\infty^{(1)} \psi = \tilde\phi$ and $\mathrm{pr}_\infty^{(2)} \psi = \phi_m$.
  Since $\alpha$ is semiprojective, there are $n \geq m$ and a $^*$-homomorphism $\psi_n \colon A_0 \rightarrow P_n$ such that $\rho_{\infty, n}^P \psi_n = \psi \alpha$.
  Define $\tilde\phi_n = \mathrm{pr}_n^{(1)} \psi_n$.
  Then
  \begin{equation}
    \rho_{\infty, n}^E\tilde{\phi}_n = \rho_{\infty, n}^E \mathrm{pr}_n^{(1)} \psi_n = \mathrm{pr}_{\infty}^{(1)} \rho_{\infty, n}^P \psi_n = \mathrm{pr}_{\infty}^{(1)} \psi \alpha = \tilde{\phi} \alpha.
  \end{equation}
  Furthermore,
  \begin{align}
  	\begin{split}
      \pi_n \tilde\phi_n 
      &= \pi_n \mathrm{pr}^{(1)}_n \psi_n \\
      &= \rho_{n, m}^B \mathrm{pr}_n^{(2)} \psi_n \\ 
      &= \rho_{n, m}^B \mathrm{pr}_\infty^{(2)} \rho_{\infty, n}^P \psi_n \\
      &= \rho_{n, m}^B \mathrm{pr}_\infty^{(2)} \psi \alpha \\ 
      &= \rho_{n, m}^B \phi_m \alpha.
    \end{split}
  \end{align}
  so $\tilde\phi_n$ satisfies the required properties.
\end{proof}

The following two lemmas and their proofs are standard and well-known to experts, but we could not find precise references, so we include the proofs here.

\begin{lemma}
  \label{lem:perturb}
  Suppose $A_0$ and $A$ are separable $C^*$-algebras and $\alpha \colon A_0 \rightarrow A$ is a semiprojective $^*$-homomorphism.
  For every finite set $\mathcal G \subseteq A_0$ and $\delta > 0$, there is a finite set $\mathcal H \subseteq A$ and $\gamma > 0$ such that if $B$ is a $C^*$-algebra and $\phi \colon A \rightarrow B$ is a self-adjoint linear map with $\|\phi(a_1a_2) - \phi(a_1)\phi(a_2)\| < \gamma$ for all $a_1,a_2 \in \mathcal H$, then there is a $^*$-homomorphism $\phi' \colon A_0 \rightarrow B$ such that $\|\phi'(a) - \phi(\alpha(a))\| < \delta$ for all $a \in \mathcal G$.
\end{lemma}

\begin{proof}
  Suppose the result fails for some finite set $\mathcal G \subseteq A$ and $\delta > 0$.
  Let $(\mathcal H_n)_{n=1}^\infty$ be an increasing sequence of finite subsets of $A$ with dense union and set $\gamma_n = 1/n$.
  For each $n \geq 1$, there are a $C^*$-algebra $B_n$ and a self-adjoint linear map $\phi_n \colon A \rightarrow B_n$ such that $\|\phi_n(a_1a_2) - \phi_n(a_1)\phi_n(a_2)\|< \gamma_n$ for all $a_1, a_2 \in \mathcal H_n$ but such that there is no $^*$-homomorphism $\phi' \colon A \rightarrow B_n$ with $\|\phi'(a) - \phi_n(\alpha(a))\| < \delta$ for all $a \in \mathcal G$.
	
  Let $B = \prod_{n=1}^\infty B_n / \bigoplus_{n=1}^\infty B_n$.
  Then the sequence $(\phi_n)_{n=1}^\infty$ induces a $^*$-ho\-mo\-morph\-ism $\phi \colon A \rightarrow B$.
  Realizing $B$ as an inductive limit of the $C^*$-algebras $\prod_{m=n}^\infty B_m$ for $n\geq1$, with the connecting maps given by the canonical projections, there are an integer $n \geq 1$ and a $^*$-ho\-mo\-morph\-ism $\phi' \colon A_0 \rightarrow \prod_{m=n}^\infty B_m$ such that $\phi'$ lifts $\phi \alpha$, since $\alpha$ is semiprojective.
  If we denote the components of $\phi'$ by $\phi'_m \colon A_0 \rightarrow B_m$ for $m \geq n$, then $\lim_m \|\phi'_m(a) - \phi_m(\alpha(a))\| = 0$ for all $a \in A_0$.
  This contradicts the choice of $\phi_m$ for some sufficiently large $m \geq n$.
\end{proof}

The following consequence of Lemma~\ref{lem:perturb} and Theorem~\ref{thm:liftable-closed} will be used in the proof of our relative homotopy lifting result in Theorem~\ref{thm:hlp-qdl}.

\begin{lemma}
  \label{lem:approx-lift}
  Suppose $A_0$, $A$, $B$, and $E$ are $C^*$-algebras with $A_0$ and $A$ separable.
  Let $\alpha \colon A_0 \rightarrow A$ be a semiprojective $^*$-ho\-mo\-mor\-phism, $\pi \colon E \rightarrow B$ be a surjective $^*$-ho\-mo\-mor\-phism, and $\phi \colon A \rightarrow B$ be a $^*$-homomorphism.
  If there is an approximate morphism $(\sigma_\lambda) \colon A \rightarrow E$ such that $\lim_\lambda \|\pi(\sigma_\lambda(a)) - \phi(a)\| = 0$ for all $a \in A$, then there is a $^*$-homomorphism $\tilde \phi \colon A_0 \rightarrow E$ such that $\pi \tilde \phi = \phi \alpha$.
\end{lemma}

\begin{proof}
  By \cite[Lemma~1.2]{top-e}, there are a separable $C^*$-algebra $A_1$ and semiprojective $^*$-ho\-mo\-mor\-phisms $\alpha_0 \colon A_0 \rightarrow A_1$ and $\alpha_1 \colon A_1 \rightarrow A$ such that $\alpha = \alpha_1 \alpha_0$.
  Apply Theorem~\ref{thm:liftable-closed} to the semiprojective $^*$-homomorphism $\alpha_0$ with $\mathcal F = \emptyset$ and $\epsilon = 1$ to obtain the corresponding finite set $\mathcal G \subseteq A_1$ and $\delta > 0$.
  Then apply Lemma~\ref{lem:perturb} with $\alpha_1$ in place of $\alpha$ and $\delta/2$ in place of $\delta$ to obtain a corresponding finite set $\mathcal H \subseteq A$ and $\gamma > 0$.
  Choose $\lambda$ such that for all $a_1, a_2 \in \mathcal{H}$, we have $\|\sigma_\lambda(a_1a_2) - \sigma_\lambda(a_1) \sigma_\lambda(a_2)\| < \gamma$ and $\|\pi(\sigma_\lambda(a)) - \phi(a)\| < \delta/2$ for all $a \in \mathcal G$.
  By the choice of $\mathcal H$ and $\gamma$, there is a
  $^*$-homomorphism $\sigma' \colon A_1 \rightarrow E$ such that
  $\|\sigma'(a) - \sigma_\lambda(\alpha_1(a))\| < \delta/2$ for all
  $a\in \mathcal{G}$.
  Now, we have $\|\pi(\sigma'(a)) - \phi(\alpha_1(a))\| < \delta$ for all $a \in \mathcal G$.
  By the choice of $\mathcal G$ and $\delta$, there is a $^*$-homomorphism $\tilde \phi \colon A_0 \rightarrow E$ such that $\pi \tilde\phi = \phi\alpha_1\alpha_0 = \phi\alpha$.
\end{proof}

\section{A relative homotopy lifting property}\label{sec:hlp}

We now recall the formal definition of the homotopy lifting property described in the introduction and introduce a relative version of the property for $^*$-ho\-mo\-morph\-isms.

\begin{definition}\label{def:hlp}
  For $C^*$-algebras $A_0$, $A$, $B$, and $E$ and $^*$-homomorphisms $\alpha \colon A_0 \rightarrow A$ and $\pi \colon E \rightarrow B$ with $\pi$ surjective, we say the pair $(\alpha, \pi)$ satisfies the \emph{homotopy lifting property} if the diagram completion problem
  \begin{equation}
    \begin{tikzcd}
      A  \arrow[bend 	left]{ddrrr}{\tilde\phi} \arrow[bend right]{dddrr}[swap]{\theta}&[-10pt] &[-10pt] &[15pt] \\[-10pt]
      & A_0 \arrow{ul}[swap]{\alpha} \arrow[dashed]{dr}{\tilde\theta} & & \\[-10pt]
      & & IE \arrow{r}{{\rm ev}_0^E} \arrow{d}[swap]{I\pi} & E \arrow{d}{\pi} \\[15pt]
      & & IB \arrow{r}[swap]{{\rm ev}_0^B} & B 
    \end{tikzcd}
  \end{equation}
  always has a solution $\tilde\theta$, where each arrow represents a $^*$-homomorphism.
  We say $(A, \pi)$ satisfies the \emph{homotopy lifting property} if $(\mathrm{id}_A, \pi)$ does.
\end{definition}

The following result of Blackadar is the main result of \cite{Blackadar16}. It provides large classes of examples of $C^*$-algebras for which the homotopy lifting property holds.

\begin{theorem}[{\cite[Theorem~5.1]{Blackadar16}}]
  \label{thm:hlp-sp}
  If $A$ is a separable semiprojective $C^*$-algebra and $\pi$ is a surjective $^*$-homomorphism, then $(A, \pi)$ satisfies the homotopy lifting property.
\end{theorem}

We have been unable to prove a relative version of Theorem~\ref{thm:hlp-sp} for semiprojective $^*$-homomorphisms. If such a result is true, then the asymptotic homotopy lifting property would hold for all separable $C^*$-algebras and surjective $^*$-homomorphisms via our proof of Theorem~\ref{thm:ahlp-asp} (see Theorem~\ref{thm:ahlp}).
The following relative homotopy lifting property with respect to approximate $^*$-homomorphisms will lead to Theorem~\ref{thm:ahlp-qd}.

\begin{theorem}\label{thm:hlp-qdl}
  Let $A_0$, $A$, $B$, and $E$ be $C^*$-algebras with $A_0$ and $A$ separable.
  If $\alpha \colon A_0 \rightarrow A$ is a semiprojective $^*$-homomorphisms and $\pi \colon E \rightarrow B$ is a surjective approximately decomposable $^*$-homomorphism, then $(\alpha, \pi)$ satisfies the homotopy lifting property.
\end{theorem}

\begin{proof}
  Consider $^*$-homomorphisms $\theta \colon A \rightarrow IB$ and $\tilde\phi \colon A \rightarrow E$ with $\pi \tilde\phi = \mathrm{ev}_0^B \theta$.
  Let $Z_\pi = \{ (e, f) \in E\oplus IB : \pi(e) = f(0) \}$ be the mapping cylinder of $\pi$.
  The map $\varpi \colon IE\to Z_\pi$ given by $\varpi(g) = ( g(0), \pi g )$ is a $^*$-homomorphism.  Further, note that $\varpi$ is surjective.  Indeed, if $(e, f) \in Z_\pi$, then the surjectivity of $I\pi$ implies there exists $\tilde f \in IE$ such that $\pi\tilde f = f$.  Define $g \in IE$ by 
  \begin{equation}
  	g(s) = (1 - s) \big(e - g(0)\big) + g(t), \quad s \in I.
  \end{equation}
  Then $\pi g = f$ and $g(0) = e$, so $\varpi(g) = (e, f)$, as required.  By Lemma~\ref{lem:approx-lift}, it is enough to construct an approximate morphism $(\tilde{\theta}_\lambda)\colon A\to IE$ with \begin{equation}\label{eq:rhlp-target}
  	\lim_\lambda \big\|\varpi\big(\tilde\theta_\lambda (a)\big) - \big(\tilde{\phi}(a), \theta(a)\big)\big\| = 0, \quad a \in A.
  \end{equation}
  Indeed, in this case, $(\tilde \phi \oplus \theta)\alpha$ lifts along $\varpi$ to a $^*$-homomorphism $\tilde\theta \colon A_0 \rightarrow IE$, which necessarily satisfies $\mathrm{ev}_0^E \tilde\theta = \tilde\phi\alpha$ and $(I\pi)\tilde\theta = \theta\alpha$.
	
  To this end, let $\iota \colon \ker(\pi) \rightarrow E$ be the inclusion map and use Proposition~\ref{prop:approx-direct-sum} to obtain approximate morphisms $(\sigma_\lambda)\colon B\to E$ and $(\rho_\lambda)\colon E\to \ker(\pi)$ that are self-adjoint, linear, pointwise-bounded, and satisfy
  \begin{equation}
    \rho_\lambda \iota = \id_{\ker(\pi)}, \quad
    \pi\sigma_\lambda = \id_B, \quad\text{and} \quad
    \iota \rho_\lambda + \sigma_\lambda \pi = \id_E
  \end{equation}
  for every $\lambda$.
  Define $\tilde{\vartheta}_\lambda\colon A\to IE$ by
  \begin{equation}
    \tilde{\vartheta}_\lambda(a)(s)
    = \rho_\lambda \big(\tilde{\phi}(a)\big) +
    \sigma_\lambda\big( \theta(a)(s)\big), \quad a \in A,\ s \in I.
  \end{equation}
  For all $a\in A$ and for all $\lambda$, we have
  \begin{align}
    \begin{split}\label{eq:eval0}
      \tilde{\vartheta}_\lambda(a)(0)
      &= \rho_\lambda\big( \tilde{\phi}(a) \big) +
      \sigma_\lambda\big( \pi\big(\tilde{\phi}(a)\big) \big)\\
      &= (\iota\rho_\lambda + \sigma_\lambda \pi)\big(
        \tilde{\phi}(a)\big)\\
      &= \tilde{\phi}(a).
    \end{split}
  \end{align}
  Moreover,
  \begin{equation}
    \pi\big( \tilde\vartheta_\lambda(a)(s) \big) = \pi\big(
    \sigma_\lambda\big(\theta(a)(s)\big) \big) = \theta(a)(s), \quad a \in A,\ s \in I,\ \lambda.
  \end{equation}
  Therefore, $\varpi \tilde{\vartheta}_\lambda = \tilde{\phi} \oplus \theta$ for all $\lambda$.
	
  The proof would be finished if we proved that $(\tilde{\vartheta}_\lambda)$ is an approximate morphism.
  Unfortunately, this may not be the case, and we must work a bit harder to obtain an approximate morphism that does the job.
  To this end, use Proposition~\ref{prop:ptwise-bounded}\ref{ptbd3} to obtain an approximate morphism $(\sigma'_\lambda) \colon B \rightarrow E$, indexed by the same directed set, that is pointwise-bounded, equicontinuous, and satisfies $\lim_\lambda \| \sigma'_\lambda (b) - \sigma_\lambda (b) \| = 0$ for all $b\in B$.
  Define $\tilde\theta_\lambda\colon A\to IE$ by
  \begin{equation}
    \tilde\theta_\lambda(a)(s)
   = \rho_\lambda\big(
      \tilde{\phi}(a) \big) + \sigma'_\lambda\big( \theta(a)(s) \big),
    \quad a \in A,\ s \in I,
  \end{equation}
  and note that
  \begin{equation}\label{eq:1}
    \lim_\lambda \| \tilde\theta_\lambda
    (a)(s) -
      \tilde\vartheta_\lambda (a)(s) \| = 0, \quad 
     a \in A,\ s \in I.
  \end{equation}
  
  Fix $a_1, a_2 \in A$.
  Let $f_\lambda \in I\mathbb C$ be given by
  \begin{equation}
    f_\lambda(s) = \| \tilde{\theta}_\lambda(a_1a_2)(s) - \tilde{\theta}_\lambda(a_1)(s) \tilde{\theta}_\lambda(a_2)(s) \|, \quad s \in I.
  \end{equation}
  Then $(f_\lambda)$ is uniformly bounded, equicontinuous, and converges pointwise to $0$. The Arzel\`a--Ascoli Theorem implies that $(f_\lambda)$ is contained in a $\|\cdot\|$-compact subspace of $I\mathbb C$.
  Since every $\|\cdot\|$-convergent subnet of $(f_\lambda)$ converges to 0, it follows that $(f_\lambda)$ converges uniformly to 0.
  Therefore,
  \begin{equation}
    \lim_\lambda \|\tilde\theta_\lambda(a_1 a_2) - \tilde\theta_\lambda(a_1)	\tilde\theta_\lambda(a_2) \| = 0, \quad a_1, a_2 \in A.
  \end{equation}
  A similar argument shows that
  \begin{equation}
    \lim_\lambda \big\|\tilde\theta_\lambda(r_1a_1 + s_2a_2) -	\big(r_1\tilde\theta_\lambda(a_1) + r_2\tilde\theta_\lambda(a_2)\big)	\big\| = 0, \quad r_1, r_2 \in \mathbb C,\ a_1, a_2 \in A,
  \end{equation}
  and also that
  \begin{equation}
    \lim_\lambda \|\tilde\theta_\lambda(a^*) -	\tilde\theta_\lambda(a)^* \| = 0, \quad a \in A.
  \end{equation}
  The last three equations together imply that $(\tilde\theta_\lambda)$ is an approximate morphism.
  
  Note that \eqref{eq:1} implies that
  \begin{equation}
    \lim_\lambda \big\| \pi\big( \tilde\theta_\lambda(a)(s) \big) - \theta(a)(s) \big\| = 0, \quad a \in A,\ s \in I.
  \end{equation}
  Another application of equicontinuity shows that this convergence is uniform in $s$, and so 
  \begin{equation}\label{eq:second-coord}
  	\lim_\lambda \big\|(I\pi)\big( \tilde\theta_\lambda(a)\big) - \theta(a) \big\| = 0, \quad a \in A.
  \end{equation}
  Further, $\lim_\lambda \|\tilde\theta_\lambda(a)(0) - \tilde\phi(a)\| = 0$ for all $a \in A$ by \eqref{eq:eval0} and \eqref{eq:1}.  This together with \eqref{eq:second-coord} implies \eqref{eq:rhlp-target}, as required.
\end{proof}

\section{The asymptotic homotopy lifting property}\label{sec:ahlp}

Without further ado, we formally introduce the asymptotic homotopy lifting property discussed in the introduction.

\begin{definition}\label{def:ahlp}
  For $C^*$-algebras $A$, $B$, and $E$ and a surjective $^*$-homomorphism $\pi \colon E \rightarrow B$, we say the pair $(A, \pi)$ has the \emph{asymptotic homotopy lifting property} if the diagram completion problem
  \begin{equation}
    \begin{tikzcd}
      A \arrow[phantom]{ddr}[description]{\scriptstyle\cong}
      \arrow[phantom]{rrd}[description]{\scriptstyle\cong}
      \arrow[dashed]{dr}[description]{\approx}{\tilde\theta}\arrow[bend right]{ddr}[description]{\approx}[swap]{\theta} \arrow[bend left]{drr}[description]{\approx}{\tilde\phi} & &[10pt] \\ & IE 
      \arrow{d}[swap]{I\pi} \arrow{r}{{\rm ev}_0^E} & E \arrow{d}{\pi} \\[10pt] & IB \arrow{r}[swap]{{\rm ev}_0^B} & B 
    \end{tikzcd}
  \end{equation}
  always has a solution; i.e., given asymptotic morphisms $\tilde\phi \colon A \xrightarrow\approx E$ and $\theta \colon A \xrightarrow\approx IB$ with $\mathrm{ev}^B_0 \theta \cong \pi \tilde\phi$, there is an asymptotic morphism $\tilde\theta \colon A \xrightarrow\approx IE$ with $\mathrm{ev}^E_0 \tilde\theta \cong \tilde\phi$ and $(I\pi)\tilde \theta \cong \theta$.
\end{definition}

Theorems~\ref{thm:ahlp-asp} and~\ref{thm:ahlp-qd} will both be derived as corollaries of Theorem~\ref{thm:ahlp} below, which reduces the asymptotic homotopy lifting property to the relative homotopy lifting property for genuine $^*$-homomorphisms.
The proof is deferred to the next section, but we will show here how deduce the main theorems in the introduction from this theorem.

\begin{theorem}\label{thm:ahlp}
  Suppose $A$, $B$, and $E$ are $C^*$-algebras with $A$ separable and $\pi \colon E \rightarrow B$ is a surjective $^*$-homomorphism.
  Fix a shape system $(\underline A, \underline \alpha)$ for $A$.
  If $(\alpha_n, \pi)$ and $(\alpha_n, I\pi)$ satisfy the homotopy lifting property for all $n \geq 1$, then $(A, \pi)$ satisfies the asymptotic homotopy lifting property.
\end{theorem}

Combining this reduction theorem with Blackadar's homotopy lifting theorem (Theorem~\ref{thm:hlp-sp}) gives Theorem~\ref{thm:ahlp-asp}.

\begin{proof}[Proof of Theorem~\ref{thm:ahlp-asp}]
	Suppose $(\underline A, \alpha)$ is an inductive system of semiprojective $C^*$-algebras with limit $A$ and $\pi \colon E \rightarrow B$ is a surjective $^*$-homomorphism between $C^*$-algebras $B$ and $E$.  Then $I \pi \colon IE \rightarrow IB$ is also a surjective $^*$-homomorphism. By Theorem~\ref{thm:hlp-sp}, $(A_n, \pi)$ and $(A_n, I\pi)$ satisfy the homotopy lifting property for all $n \geq 1$, and it follows that $(\alpha_n, \pi)$ and $(\alpha_n, I\pi)$ satisfy the homotopy lifting property for all $n \geq 1$.  By Theorem~\ref{thm:ahlp}, the pair $(A, \pi)$ satisfies the asymptotic homotopy lifting property.
\end{proof}

Using essentially the same proof but quoting Theorem~\ref{thm:hlp-qdl} in place of Theorem~\ref{thm:hlp-sp} yields the following result.
Note that this implies Theorem~\ref{thm:ahlp-qd}, as quasidiagonal extensions are approximately decomposable (Proposition~\ref{prop:qd-qdl}).

\begin{corollary}\label{cor:ahlp-qdl}
	If $A$, $B$, and $E$ are separable $C^*$-algebras and $\pi \colon E \rightarrow B$ is a surjective approximately decomposable $^*$-homomorphism, then $(A, \pi)$ satisfies the asymptotic homotopy lifting property.
\end{corollary}

\begin{proof}
	First we show that $I\pi \colon I E \rightarrow I B$ is approximately decomposable.
	Using that $\pi$ is approximately decomposable, there is an approximate morphism $(\sigma_\lambda) \colon B \rightarrow E$, indexed over a directed set $\Lambda$, as in Definition~\ref{def:qdelicious}.
	By Lemma~\ref{lem:qdelicious-bounded}, we may further assume that $(\sigma_\lambda)$ is pointwise bounded.
	Let $\sigma_\Lambda \colon B \rightarrow E_\Lambda$ be the $^*$-homomorphism induced by $(\sigma_\lambda)$ as in the remarks after Definition~\ref{def:approx-morph}, let $\pi_\Lambda \colon E_\Lambda \rightarrow B_\Lambda$ be the $^*$-homomorphism given by applying $\pi$ in each coordinate, and let $\iota^E_\Lambda \colon E \rightarrow E_\Lambda$ be the canonical inclusion.
	Then $\pi_\Lambda \sigma_\Lambda = \iota^B_\Lambda$ and $\iota^E_\Lambda(e)\sigma_\Lambda(b) = \sigma_\Lambda(\pi(e)b)$ for all $b \in B$ and $e \in E$.
	
	Composing $I \sigma_\Lambda \colon IB \rightarrow IE_\Lambda$ with the natural embedding $IE_\Lambda \rightarrow (IE)_\Lambda$ produces a $^*$-homomorphisms $\bar\sigma_{\Lambda} \colon IB \rightarrow (IE)_\Lambda$ satisfying $(I\pi)_\Lambda\bar\sigma_\Lambda = \iota^{IB}_\Lambda$ and $\iota^{IE}_\Lambda(f) \bar\sigma_\Lambda(g) = \bar\sigma_\Lambda((I\pi)(f)g)$ for all $f \in IE$ and $g \in IB$.
	Any lift of $\bar\sigma_\Lambda$ to a net of functions $(\bar\sigma_\lambda) \colon IB \rightarrow IE$ will satisfy the conditions of Definition~\ref{def:qdelicious}, proving that $I\pi$ is approximately decomposable.
	
	To finish the proof, fix a shape system $(\underline A, \underline, \alpha)$ for $A$ and note that $(\alpha_n, \pi)$ and $(\alpha_n, I\pi)$ satisfy the homotopy lifting property for all $n \geq 1$ by Theorem~\ref{thm:hlp-qdl}.
	The result now follows from Theorem~\ref{thm:ahlp}.
\end{proof}

The simple idea behind the proof of Theorem~\ref{thm:ahlp} is somewhat obscured by the analytic details, so we provide a sketch.
For the purposes of this sketch we assume both $(A_n, \pi)$ and $(A_n, I\pi)$ satisfy the homotopy lifting property (which is the case if $A_n$ is semiprojective by Theorem~\ref{thm:hlp-sp}). This extra hypothesis allows us to suppress the index shift in the proof of Theorem~\ref{thm:ahlp}, which adds no extra technical difficulty but introduces slightly distracting notation.
Consider asymptotic morphisms $\tilde \phi \colon A \xrightarrow\approx E$ and $\theta \colon A \xrightarrow\approx IB$ satisfying $\mathrm{ev}_0^B \theta \cong \pi \tilde\phi$.
We need to construct an asymptotic morphism $\tilde\theta \colon A \xrightarrow\approx IE$ with $\mathrm{ev}_0^E \tilde\theta \cong \tilde \phi$ and $(I\pi)\tilde\theta \cong \theta$.

Fix a shape system $(\underline A, \underline \alpha)$ for $A$.  We will construct $\tilde\theta$ via a diagrammatic representation $(\underline{\tilde \theta}, \underline{\tilde k}, \underline t) \colon (\underline A, \underline \alpha) \rightarrow IE$.  The most difficult part of the proof is constructing compatible diagrammatic representations for $\tilde\phi$ and $\theta$, which is the content of Lemma~\ref{lem:somehow}.  This will provide diagrammatic representations 
\begin{equation} 
	(\underline{\tilde \phi}, \underline{\tilde h}, \underline t) \colon (\underline A, \underline \alpha) \rightarrow E \quad \text{and} \quad (\underline \theta, \underline k, \underline t) \colon (\underline A, \underline \alpha) \rightarrow IB
\end{equation}
satisfying $(\mathrm{ev}_0^E)_*(\underline \theta, \underline k, \underline t) = \pi_*(\underline{\tilde \phi}, \underline{\tilde h}, \underline t)$.  Let $(\underline \phi, \underline h, \underline t) = \pi_*(\underline{\tilde\phi}, \underline{\tilde h}, \underline t)$.  

We prefer to visualize this data as in the diagram below:
\begin{equation}\label{eq:AB-picture}
\begin{aligned}
	\begin{tikzpicture}[scale=1,
	farrow/.style={decoration={markings,
			mark=at position 0.5 with {\arrow{Computer Modern Rightarrow}}}}
	]
	
	\tikzstyle{every node}=[font=\footnotesize]
	
	\def\mycolsep{2}
	\def\myrowsep{1.5}
	\def\r{1.5pt}
	\def\eps{1pt}
	
	
	\def\leftsquare{ (0,0) -- (\mycolsep,0) -- (\mycolsep, \myrowsep) --
		(0, \myrowsep) -- cycle };
	\fill[pattern color=black!20,pattern=north east lines] \leftsquare;
	
	\fill (0,0) [black] circle (\r) node (l1) {};
	\fill (\mycolsep,0) [black] circle (\r) node (l2) {};
	\fill (\mycolsep, \myrowsep) [black] circle (\r) node (l3) {};
	\fill (0, \myrowsep) [black] circle (\r) node (l4) {};
	
	\node[below left] (ll1) at (0,0) {$ \phi_1 $};
	\node[below right] (ll2) at (\mycolsep,0) {$ \phi_2\alpha_1 $};
	
	\begin{scope}[
		every path/.style={
			postaction={nomorepostaction,farrow,decorate}
		}
		]
		
		\draw (l1) -- (l2) node[midway, below] {$ h_1 $};
		\draw (l2) -- (l3) node[midway, right] {$ \theta_2\alpha_1 $};
		\draw (l1) -- (l4) node[midway, left] {$ \theta_1 $};
		\draw (l4) -- (l3); 
	\end{scope}
	
	\node[fill=white, rounded corners] at (.5*\mycolsep, .5*\myrowsep)
	{$ k_1 $};
	
	
	\def\sqsep{4}
	
	\def\rightsquare{ (\sqsep,0) -- (\mycolsep+\sqsep,0) -- (\mycolsep+\sqsep, \myrowsep) -- (\sqsep, \myrowsep) -- cycle };
	\fill[pattern color=black!20,pattern=north east lines] \rightsquare;
	
	\fill (\sqsep,0) [black] circle (\r) node (r1) {};
	\fill (\mycolsep+\sqsep,0) [black] circle (\r) node (r2) {};
	\fill (\mycolsep+\sqsep, \myrowsep) [black] circle (\r) node (r3) {};
	\fill (\sqsep, \myrowsep) [black] circle (\r) node (r4) {};
	
	\node[below left] (rr1) at (\sqsep,0) {$ \phi_2 $};
	\node[below right] (rr2) at (\mycolsep+\sqsep,0) {$ \phi_3\alpha_2 $};
	
	\begin{scope}[
		every path/.style={
			postaction={nomorepostaction,farrow,decorate}
		}
		]
		
		\draw (r1) -- (r2) node[midway, below] {$ h_2 $};
		\draw (r2) -- (r3) node[midway, right] {$ \theta_3\alpha_2 $};
		\draw (r1) -- (r4) node[midway, left] {$ \theta_2 $};
		\draw (r4) -- (r3); 
	\end{scope}
	
	\node[fill=white, rounded corners] at (.5*\mycolsep+\sqsep, .5*\myrowsep)
	{$ k_2 $};
	
	
	\node (x1) at (.5*\mycolsep, 1.6*\myrowsep) {$ \text{Hom}(A_1, B) $};
	\node (x2) at (.5*\mycolsep+\sqsep, 1.6*\myrowsep) {$  \text{Hom}(A_2, B) $};
	\node (x3) at (.5*\mycolsep+1.75*\sqsep, 1.6*\myrowsep) {$ \cdots $};
	
	\begin{scope}[arrows={-Computer Modern Rightarrow}]
		\draw (x3) -- (x2) node[midway, above] {$ \alpha_2^* $};
		\draw (x2) -- (x1) node[midway, above] {$ \alpha_1^* $};;
	\end{scope}
	
	\node (x3) at (.5*\mycolsep+1.75*\sqsep, .5*\myrowsep) {$ \cdots $};
	
	
	\begin{scope}[on background layer]
		\def\funkypath{ (r4) .. controls (1.5*\mycolsep, .75*\myrowsep)
			.. (l3) -- (l2) .. controls (1.5*\mycolsep, .25*\myrowsep)
			.. (r1) -- cycle }
		
		\def\funkypath{ (r4.center) .. controls (1.5*\mycolsep,
			.75*\myrowsep) .. (l3.center) -- (l2.center) .. controls
			(1.5*\mycolsep, .25*\myrowsep) .. (r1.center) -- cycle}
		
		\fill[gray!20] \funkypath;
		
		\draw[arrows={-Computer Modern Rightarrow}, gray!80]
		(1.54*\mycolsep, .7*\myrowsep) -- (1.46*\mycolsep, .7*\myrowsep);
		\draw[arrows={-Computer Modern Rightarrow}, gray!80]
		(1.54*\mycolsep, .5*\myrowsep) -- (1.46*\mycolsep, .5*\myrowsep);
		\draw[arrows={-Computer Modern Rightarrow}, gray!80]
		(1.54*\mycolsep, .3*\myrowsep) -- (1.46*\mycolsep, .3*\myrowsep);
	\end{scope}
\end{tikzpicture}
\end{aligned}
\end{equation}
Here, we regard the sets $\mathrm{Hom}(A_n, B)$ as topological spaces with the point-norm topology.
Moreover, the maps $h_n, \theta_n \colon A_n \rightarrow IB$ and $k_n \colon A_n \rightarrow I(IB)$ are identified with continuous functions $I \rightarrow \mathrm{Hom}(A_n, B)$ and $I^2 \rightarrow \mathrm{Hom}(A_n, B)$.
The bridge between the squares is meant to suggest that the left edge of square $n+1$ is being mapped to the right edge of square $n$ and the asymptotic morphism $A \xrightarrow\approx IB$ is essentially given by gluing these squares together via this identification.

Now consider a similar picture with $E$ in place of $B$, where the bottom edges of the squares are given by the diagrammatic representation of $\tilde\phi$:
\begin{equation}\label{eq:AE-picture}
\begin{aligned}
\begin{tikzpicture}[scale=1,
	farrow/.style={decoration={markings,
			mark=at position 0.5 with {\arrow{Computer Modern Rightarrow}}}}
	]
	
	\tikzstyle{every node}=[font=\footnotesize]
	
	\def\mycolsep{2}
	\def\myrowsep{1.5}
	\def\r{1.5pt}
	\def\eps{1pt}
	
	
	\tikzdeclarepattern{name=north east stripes,
		type=uncolored,
		bottom left={(-.1pt,-.1pt)},
		top right={(10.1pt,10.1pt)},
		tile size={(10pt,10pt)},
		tile transformation={rotate=45},
		code={
			\tikzset{x=1pt,y=1pt}
			\draw[thin]  (0,2.5) -- (5,2.5) (5,7.5) -- (10,7.5); 
	} }
	
	\def\leftsquare{ (0,0) -- (\mycolsep,0) -- (\mycolsep, \myrowsep) --
		(0, \myrowsep) -- cycle };
	\fill[pattern color=black!20,pattern=north east stripes] \leftsquare;
	
	\fill (0,0) [black] circle (\r) node (l1) {};
	\fill (\mycolsep,0) [black] circle (\r) node (l2) {};
	\draw (\mycolsep, \myrowsep) [black] circle (\r) node (l3) {};
	\draw (0, \myrowsep) [black] circle (\r) node (l4) {};
	
	\node[below left] (ll1) at (0,0) {$ \tilde\phi_1 $};
	\node[below right] (ll2) at (\mycolsep,0) {$ \tilde\phi_2\alpha_1 $};
	
	\begin{scope}[
		every path/.style={
			postaction={nomorepostaction,farrow,decorate}
		}
		]
		
		\draw (l1) -- (l2) node[midway, below] {$ \tilde h_1 $};
		\draw [dashed] (l2) -- (l3) node[midway, right] {$ \tilde\theta_2\alpha_1 $};
		\draw [dashed] (l1) -- (l4) node[midway, left] {$ \tilde\theta_1 $};
		\draw [dashed] (l4) -- (l3); 
	\end{scope}
	
	\node[fill=white, rounded corners] at (.5*\mycolsep, .5*\myrowsep)
	{$ \tilde k_1 $};
	
	
	\def\sqsep{4}
	
	\def\rightsquare{ (\sqsep,0) -- (\mycolsep+\sqsep,0) -- (\mycolsep+\sqsep, \myrowsep) -- (\sqsep, \myrowsep) -- cycle };
	\fill[pattern color=black!20,pattern=north east stripes] \rightsquare;
	
	\fill (\sqsep,0) [black] circle (\r) node (r1) {};
	\fill (\mycolsep+\sqsep,0) [black] circle (\r) node (r2) {};
	\draw (\mycolsep+\sqsep, \myrowsep) [black] circle (\r) node (r3) {};
	\draw (\sqsep, \myrowsep) [black] circle (\r) node (r4) {};
	
	\node[below left] (rr1) at (\sqsep,0) {$ \tilde\phi_2 $};
	\node[below right] (rr2) at (\mycolsep+\sqsep,0) {$ \tilde\phi_3\alpha_2 $};
	
	\begin{scope}[
		every path/.style={
			postaction={nomorepostaction,farrow,decorate}
		}
		]
		
		\draw (r1) -- (r2) node[midway, below] {$ \tilde h_2 $};
		\draw [dashed] (r2) -- (r3) node[midway, right] {$ \tilde\theta_3\alpha_2 $};
		\draw [dashed] (r1) -- (r4) node[midway, left] {$ \tilde\theta_2 $};
		\draw [dashed] (r4) -- (r3); 
	\end{scope}
	
	\node[fill=white, rounded corners] at (.5*\mycolsep+\sqsep, .5*\myrowsep)
	{$ \tilde k_2 $};
	
	
	\node (x1) at (.5*\mycolsep, 1.6*\myrowsep) {$ \text{Hom}(A_1, E) $};
	\node (x2) at (.5*\mycolsep+\sqsep, 1.6*\myrowsep) {$  \text{Hom}(A_2, E) $};
	\node (x3) at (.5*\mycolsep+1.75*\sqsep, 1.6*\myrowsep) {$ \cdots $};
	
	\begin{scope}[arrows={-Computer Modern Rightarrow}]
		\draw (x3) -- (x2) node[midway, above] {$ \alpha_2^* $};
		\draw (x2) -- (x1) node[midway, above] {$ \alpha_1^* $};;
	\end{scope}
	
	\node (x3) at (.5*\mycolsep+1.75*\sqsep, .5*\myrowsep) {$ \cdots $};
	
	
	\begin{scope}[on background layer]
		\def\funkypath{ (r4) .. controls (1.5*\mycolsep, .75*\myrowsep)
			.. (l3) -- (l2) .. controls (1.5*\mycolsep, .25*\myrowsep)
			.. (r1) -- cycle }
		
		\def\funkypath{ (r4.center) .. controls (1.5*\mycolsep,
			.75*\myrowsep) .. (l3.center) -- (l2.center) .. controls
			(1.5*\mycolsep, .25*\myrowsep) .. (r1.center) -- cycle}
		
		\fill[gray!20] \funkypath;
		
		\draw[arrows={-Computer Modern Rightarrow}, gray!80]
		(1.54*\mycolsep, .7*\myrowsep) -- (1.46*\mycolsep, .7*\myrowsep);
		\draw[arrows={-Computer Modern Rightarrow}, gray!80]
		(1.54*\mycolsep, .5*\myrowsep) -- (1.46*\mycolsep, .5*\myrowsep);
		\draw[arrows={-Computer Modern Rightarrow}, gray!80]
		(1.54*\mycolsep, .3*\myrowsep) -- (1.46*\mycolsep, .3*\myrowsep);
	\end{scope}
\end{tikzpicture}
\end{aligned}
\end{equation}
The goal is to define the maps $\tilde\theta_n$ and $\tilde k_n$ completing the schematic in \eqref{eq:AE-picture} so that \eqref{eq:AE-picture} is a lift of \eqref{eq:AB-picture} along the quotient map $\pi$.  This will be done in two stages.  First, the homotopy lifting property for $(A_n, \pi)$ provides a $^*$-homomorphism $\tilde \theta_n \colon A_n \rightarrow IE$ such that $\mathrm{ev}_0^E\tilde\theta_n = \tilde\phi_n$ and $(I\pi) \tilde\theta_n= \theta_n$.  This fills in the left (and hence also the right) side of each square in \eqref{eq:AE-picture}.  

To complete the interior and top edge of the squares in \eqref{eq:AE-picture}, let $X \subseteq I^2$ denote the subspace consisting of the left, right, and bottom edge of the square, and note that there is a homeomorphism $I^2 \rightarrow I^2$ mapping $X$ onto $I \times \{0\} \subseteq I^2$.  View $k_n$ as a $^*$-homomorphism $A_n \rightarrow I^2 B$ and let $\tilde l_n \colon A_n \rightarrow XE$ denote the $^*$-homomorphisms determined by \eqref{eq:AE-picture}.  Then $(X\pi)\tilde l_n \colon A_n \rightarrow XB$ coincides with the composition of $k_n$ with the restriction map $I^2 B \rightarrow XB$.  The homotopy lifting property of $(A_n, I\pi)$ (in the form of Lemma~\ref{lem:X}) then provides a $^*$-homomorphism $\tilde k_n \colon A_n \rightarrow I^2E$ lifting $k_n$ and $\tilde l_n$, hence completing the diagram \eqref{eq:AE-picture}.

View $\tilde k_n$ as a map $A_n \rightarrow I(IE)$.  Then the maps $\tilde \theta_n$ and $\tilde k_n$ provide a diagrammatic representation $(\underline{\tilde \theta}, \underline{\tilde k}, \underline t)$ for an asymptotic $^*$-homomorphisms $\tilde\theta \colon A \xrightarrow\approx IE$.  By construction, 
\begin{equation}
	(\mathrm{ev}_0^E)_*(\underline{\tilde\theta}, \underline{\tilde k}, \underline t) = (\underline{\tilde \phi}, \underline{\tilde h}, \underline t) \quad \text{and} \quad (I\pi)_*(\underline{\tilde\theta}, \underline{\tilde k}, \underline t)  = (\underline \theta, \underline k, \underline t),
\end{equation}
and hence $\mathrm{ev}_0^E \tilde\theta \cong \tilde\phi$ and $(I\pi)\tilde\theta \cong \theta$, so $\tilde\theta$ satisfies the required properties.

In the actual proof of Theorem~\ref{thm:ahlp} in the next section, we only know that $(\alpha_n, \pi)$ and $(\alpha_n, I\pi)$ satisfy the homotopy lifting property instead of $(A_n, \pi)$ and $(A_n, I\pi)$.
The main difference is that one must shift the index in each application of the homotopy lifting property, both for lifting the left edges of the squares and the interior of the squares.
Hence the diagram in \eqref{eq:AE-picture} will lift the diagram obtained by shifting \eqref{eq:AB-picture} two stages to the left (and removing the left two squares).
This will not affect the limiting asymptotic morphisms.
Aside from this change and taking more care with the identifications being made, the proof of Theorem~\ref{thm:ahlp} will follow the above outline closely.

\section{Proof of Theorem~\ref{thm:ahlp}}

As discussed in the sketch of the proof of Theorem~\ref{thm:ahlp} in the previous section, the bulk of the work is in constructing compatible diagram representations of $\theta$ and $\tilde \phi$.  Before tackling this problem, we isolate the following simple lemma.

\begin{lemma}\label{lem:as-exact}
  The functor $C_b(\mathbb{R}_+, \,\cdot\,)$ is exact.
\end{lemma}

\begin{proof}
  Fix an exact sequence
  \begin{equation}
    0 \longrightarrow J \overset\iota\longrightarrow E \overset\pi\longrightarrow B \longrightarrow 0
  \end{equation}
  of $C^*$-algebras, and consider the induced sequence
  \begin{equation}
    0 \longrightarrow C_b(\mathbb{R}_+, J) \overset{\bar\iota}{\longrightarrow}
    C_b(\mathbb{R}_+, E) \overset{\bar\pi}{\longrightarrow}
    C_b(\mathbb{R}_+, B) \longrightarrow 0.
  \end{equation}
  Exactness at $C_b(\mathbb{R}_+, J)$ and $C_b(\mathbb{R}_+, E)$ is clear, so it suffices to prove that $\bar\pi$ is surjective.
  To this end, fix $f\in C_b(\mathbb{R}_+, B)$.
  Let $f_n = f |_{[n,n+1]}$ for $n\geq0$.
  By the exactness of $C([n,n+1], \,\cdot\,)$, there exists $g_n' \in C([n,n+1], E)$ with $\|g_n'\| \leq \|f_n\|$ and $\pi  g_n' = f_n$.
  Define $g_n \in C([n,n+1], E)$ by
  \begin{equation}
    g_n(t) = g_n'(t) + (t-n)\big( g_{n+1}'(n) - g_n'(n+1)
    \big).    
  \end{equation}
  Then $\pi g_n = f_n$, $\| g_n \| \leq 3\|f_n\|$, and $g_n(n+1) = g_{n+1}(n+1)$ for all $n \geq 0$.
  Therefore, $(g_n)_{n=0}^\infty$ induces $g\in C_b(\mathbb{R}_+, E)$ satisfying $\bar\pi(g) = \pi g = f$.
\end{proof}

The following lemma provides the diagrammatic representations we need.
When the result is applied, we will have $D = IB$ and $\psi = \mathrm{ev}_0^B$, but the structure of $IB$ and the map $\mathrm{ev}_0^B$ play no role in the proof, and it is notationally convenient to prove this slightly more general form.

\begin{lemma}\label{lem:somehow}
  Suppose $A$, $B$, $D$, and $E$ are $C^*$-algebras with $A$ separable, $(\underline A, \underline \alpha)$ is a shape system for $A$, and $\psi \colon D \rightarrow B$ and $\pi \colon E \rightarrow B$ are $^*$-homomorphisms with $\pi$ surjective.
  If $\tilde \phi \colon A \xrightarrow\approx E$ and $\theta \colon A \xrightarrow \approx D$ are asymptotic morphisms with $\psi \theta \cong \pi \tilde\phi$, then there are diagrammatic representations $(\underline \theta, \underline k, \underline t) \colon (\underline A, \underline \alpha) \rightarrow D$ and $(\underline{\tilde\phi}, \underline{\tilde h}, \underline t) \colon (\underline A, \underline \alpha) \rightarrow E$ of $\theta$ and $\tilde\phi$ such that $\psi_*(\underline \theta, \underline k, \underline t) = \pi_*(\underline{\tilde\phi}, \underline{\tilde h}, \underline t)$.
\end{lemma}

\begin{proof}
  By Proposition~\ref{prop:shape-lift}, there is a strictly increasing sequence $(m_n)_{n=1}^\infty$ of positive integers and a morphism
  \begin{equation}\label{eq:somehow1}
    \begin{tikzcd}
      A_1 \arrow{r}{\alpha_1} \arrow{d}{\Theta_1} &
      A_2 \arrow{r}{\alpha_2} \arrow{d}{\Theta_2} &
      A_3 \arrow{r}{\alpha_3} \arrow{d}{\Theta_3} & \cdots\\
      C_b(\mathbb{R}_{\geq m_1}, D) \arrow{r}{\rho^D_1} &
      C_b(\mathbb{R}_{\geq m_2}, D) \arrow{r}{\rho^D_2} &
      C_b(\mathbb{R}_{\geq m_3}, D) \arrow{r}{\rho^D_3} & \cdots
    \end{tikzcd}
  \end{equation}
  of inductive systems inducing $\theta_{\rm as} \colon A \rightarrow
  D_{\rm as}$, where $\rho^D_n$ denotes the restriction map.
  Define $\bar \psi_n \colon C_b(\mathbb R_{\geq m_n}, D) \rightarrow C_b(\mathbb R_{\geq m_n}, B)$ by $\bar \psi_n(f)(t) = \psi(f(t))$ for all $f \in C_b(\mathbb R_{\geq m_n}, D)$, $t \in \mathbb R_{\geq m_n}$, and $n \geq 1$.

  Let $J = \ker(\pi)$ and let $\iota \colon J \rightarrow E$ denote the inclusion map.
  By Lemma~\ref{lem:as-exact}, there is an inductive system of exact sequences of $C^*$-algebras
  \begin{equation}\label{eq:somehow3}
    \begin{tikzcd}
      0 \arrow{r} &[-2.2ex] C_b(\mathbb R_{\geq m_n}, J) \arrow{r}{\bar\iota_n} \arrow{d}{\rho^J_n} & \arrow{r} C_b(\mathbb R_{\geq m_n}, E) \arrow{r}{\bar\pi_n} \arrow{d}{\rho^E_n} & C_b(\mathbb R_{\geq m_n}, B) \arrow{r} \arrow{d}{\rho^B_n} &[-2.2ex] 0\hphantom{,}\\
      0 \arrow{r} & C_b(\mathbb R_{\geq m_{n+1}}, J) \arrow{r}{\bar\iota_{n+1}} & C_b(\mathbb R_{\geq m_{n+1}}, E) \arrow{r}{\bar\pi_{n+1}} & C_b(\mathbb R_{\geq m_{n+1}}, B) \arrow{r} & 0,
    \end{tikzcd}
  \end{equation}
  where the vertical maps are the restriction maps and the horizontal maps $\bar\iota_n$ and $\bar\pi_n$ are given by applying $\iota$ and $\pi$ pointwise.
  Taking the inductive limit over $n$ yields the exact sequence
  \begin{equation}\label{eq:somehow4}
    \begin{tikzcd}
      0 \arrow{r} & J_{\rm as} \arrow{r}{\iota_{\rm as}} & E_{\rm as} \arrow{r}{\pi_{\rm as}} & B_{\rm as} \arrow{r} & 0.
    \end{tikzcd}
  \end{equation}
  
  The maps $\tilde \psi_{n+2} \Theta_{n+2} \colon A_{n+2} \rightarrow C_b(\mathbb R_{m_{n +2}}, B)$ and $\tilde \phi_{\rm as} \alpha_{\infty, n+2} \colon A_{n+2} \rightarrow E_{\rm as}$ satisfy
  \begin{equation}
  	\rho^B_{\infty, n+2} \bar\psi_{n+2} \Theta_{n+2} = \pi_{\rm as} \tilde\phi_{\rm as} \alpha_{\infty, n+2},
  \end{equation} since $\pi \tilde{\phi} \cong \psi \theta$.
  Therefore, by Proposition~\ref{prop:lift-prescribed-quotient} (enlarging the $m_n$ if necessary), there exist $^*$-homomorphisms $\tilde\Phi_{n+1} \colon A_{n+1} \rightarrow C_b(\mathbb R_{\geq m_{n+2}}, E)$ for $n \geq 1$ such that
  \begin{equation}\label{eq:somehow5}
    \begin{aligned}
      \rho^E_{\infty, n+2} \tilde\Phi_{n+1}
      & = \tilde\phi_{\rm as} \alpha_{\infty, n+1}
        \quad\text{and}\\
      \bar\pi_{n+2}
      \tilde\Phi_{n+1} & =  \bar\psi_{n+2} \Theta_{n+2} \alpha_{n+1}.
    \end{aligned}
  \end{equation}
  Note that, for all $n \geq 1$, $\rho^E_{n+2} \tilde\Phi_{n+1}$ and
  $\tilde\Phi_{n+2} \alpha_{n+1}$ are equal after applying
  $\rho^E_{\infty, n+3}$, since by \eqref{eq:somehow5},
  \begin{align}\label{eq:somehow6}
    \begin{split}
      \rho^E_{\infty, n+3} \rho^E_{n+2} \tilde\Phi_{n+1}
      &\eqrel{=} \rho^E_{\infty, n+2} \tilde\Phi_{n+1} \\
      &\eqrel{$\overset{\eqref{eq:somehow5}}{=}$}  \tilde\phi_{\rm as} \alpha_{\infty, n+1} \\
      &\eqrel{=} \tilde\phi_{\rm as} \alpha_{\infty, n+2}
        \alpha_{n+1}\\
      &\eqrel{$\overset{\eqref{eq:somehow5}}{=}$}  \rho^E_{\infty, n+3} \tilde\Phi_{n+2} \alpha_{n+1}.
    \end{split}
  \end{align}
  In particular, for all $n \geq 1$ and $a \in A_{n+1}$,
  \begin{equation}\label{eq:somehow7}
    \lim_t \big\| \tilde\Phi_{n+1}(a)(t) - \tilde\Phi_{n+2}\big(
    \alpha_{n+1}(a) \big)(t) \big\| = 0.
  \end{equation}
  
  By Lemma~\ref{lem:finite-sets}, there are finite sets $\mathcal F_n \subseteq A_n$ such that $\alpha_n(\mathcal F_n) \subseteq \mathcal F_{n+1}$ for all $n \geq 1$ and such that $\bigcup_{n=1}^\infty \alpha_{\infty, n}(\mathcal F_n)$ is dense in $A$.
  Let $(\epsilon_n)_{n=1}^\infty$ be a decreasing sequence of strictly positive real numbers with $\sum_{n=1}^\infty \epsilon_n < \infty$.
  Apply Corollary~\ref{cor:wishlist} to $\alpha_n \colon A_n \rightarrow A_{n+1}$, $\mathcal F_n \subseteq A_n$, and $\epsilon_n > 0$ to obtain a corresponding finite set $\mathcal G_{n+1} \subseteq A_{n+1}$ and tolerance $\delta_{n+1} > 0$.
  Enlarging the $\mathcal G_n$ and decreasing the $\delta_n$ if necessary, we may assume $\mathcal F_n \subseteq \mathcal G_{n+1}$ and $\epsilon_n > \delta_{n+1}$.
  Then use \eqref{eq:somehow7} to construct a strictly increasing sequence $(t_n)_{n=1}^\infty \subseteq \mathbb R$ such that, for all $n \geq 1$ and $t_n \geq m_{n+2}$,
  \begin{equation}\label{eq:somehow8}
    \big\| \tilde\Phi_{n+1}(a)(t) - \tilde\Phi_{n+2}\big(\alpha_{n+1}(a)\big)(t)\big\| < \delta_{n+1}, \qquad a \in \mathcal G_{n+1},\ t \geq t_{n+1}.
  \end{equation}
  
  Define $^*$-homomorphisms $\theta_n \colon A_n \rightarrow D$ by $\theta_n(a) = \Theta_n(a)(t_n)$ for all $a \in A_n$ and $n \geq 1$.
  Also, for $n \geq 1$, define homotopies $k_n' \colon A_n \rightarrow ID$ by
  \begin{equation}\label{eq:somehow9}
    k_n'(a)(s) = \Theta_n(a)\big((1-s)t_n + st_{n+1}\big), \qquad a \in A_n,\ s \in I.
  \end{equation} 
  Then, by construction, $(\underline \theta, \underline k', \underline t)$ is a diagrammatic representation of $\theta$.
	
  Note that if we define $\tilde\phi_n \colon A_n \rightarrow E$ and $\tilde h_n \colon A \rightarrow IE$ in the analogous way, then $(\underline{\tilde\phi}, \underline{\tilde h}, \underline t)$ need not be a diagrammatic representation of an asymptotic morphism.
  Indeed, the $^*$-homomorphisms $\tilde\phi_{n+1} \alpha_n$ and $\mathrm{ev}^E_1 \tilde h_n$ need not be equal; they are, however, point-norm close by \eqref{eq:somehow8} and the choice of the sequence $(t_n)_{n=1}^\infty$.
  We will use Corollary~\ref{cor:wishlist} to bridge the gap between these $^*$-homomorphisms via short homotopies and perturb the homotopies $k'_n$ to account for this extra path.
  In the construction below, we also start by defining a homotopy $\tilde h'_{n+1} \colon A_{n+1} \rightarrow IE$; the index shift allows us to exploit the semiprojectivity of $\alpha_n$ in the application of Corollary~\ref{cor:wishlist}.
	
  Following this strategy, for $n \geq 1$, we define a $^*$-homomorphism $\tilde\phi_n \colon A_n \rightarrow E$ by
  \begin{equation}\label{eq:somehow11}
    \tilde\phi_n(a) = \tilde\Phi_{n+1}\big(\alpha_n(a)\big)(t_n), \quad a \in A_n,
  \end{equation}
  and a homotopy $\tilde h'_{n+1} \colon A_{n+1} \rightarrow IE$ by
  \begin{equation}\label{eq:somehow12}
       \tilde h'_{n+1}(a)(s) = \tilde\Phi_{n+1}(a)\big((1-s)t_n + s t_{n+1}\big), \quad a \in A_{n+1},\ s \in I.
  \end{equation}
  Note that $\mathrm{ev}^E_0 \tilde h'_{n+1} \alpha_n = \tilde\phi_n$.
  Moreover, we have
  \begin{equation}\label{eq:somehow12.5}
    \pi \tilde\phi_n = \psi\theta_n \qquad \text{and} \qquad (I\pi) \tilde h'_{n+1}\alpha_n = (I\psi) k'_n
  \end{equation}
  by the second equation in \eqref{eq:somehow5}.  Indeed, if $a \in A_n$, then
  \begin{align}
  	\begin{split}
        \pi\big(\tilde\phi_n(a)\big) 
      &\eqrel{$\overset{\eqref{eq:somehow11}}{=}$} \pi\big(\tilde\Phi_{n+1}\big(\alpha_n(a)\big)(t_n)\big) \\
      &\eqrel{$=$} (\bar \pi_n \tilde\Phi_{n+1}\alpha_n)(a)(t_n) \\
      &\eqrel{$\overset{\eqref{eq:somehow5}}{=}$} (\bar \psi_{n+2} \Theta_{n+2} \alpha_{n+2, n})(a)(t_n) \\
      &\eqrel{$\overset{\eqref{eq:somehow1}}{=}$} \psi\big(\Theta_n(a)(t_n)\big) \\
      &\eqrel{$=$} \psi(\theta_n(a)),
    \end{split}
  \end{align}
  which proves the first equality in \eqref{eq:somehow12.5}.  For the second equality, note that for each $a \in A_n$, we have
  \begin{align}
  	\begin{split}
  	    \big((I \pi) \tilde h_{n+1}'\alpha_n\big)(a)(s)
  	  &\eqrel{$=$} \pi\big(\big(\alpha_n(a)\big)(s)\big) \\
  	  &\eqrel{$\overset{\eqref{eq:somehow12}}{=}$}  \pi\big(\tilde\Phi_{n+1}(a)\big((1-s)t_n + st_{n+1}\big)\big) \\
  	  &\eqrel{$=$} (\bar \pi_{n+2}\tilde\Phi_{n+1})(a)\big((1-s)t_n + st_{n+1}\big) \\
  	  &\eqrel{$\overset{\eqref{eq:somehow5}}{=}$} (\bar \psi_{n+2}\Theta_{n+2} \alpha_{n+2, n})(a)\big((1-s)t_n + st_{n+1}\big) \\
  	  &\eqrel{$\overset{\eqref{eq:somehow1}}{=}$} (\bar\psi_n \Theta_n)(a)\big((1-s)t_n + st_{n+1}\big) \\
  	  &\eqrel{$=$} \psi\big( \Theta_n(a)\big((1-s)t_n + s t_{n+1}\big)\big) \\
  	  &\eqrel{$\overset{\eqref{eq:somehow9}}{=}$} \psi\big(k_n'(a)(s)\big),
    \end{split}
  \end{align}
  which shows the second equality in \eqref{eq:somehow12.5}.
  
  We can also estimate $\mathrm{ev}^E_1\tilde h'_{n+1}$ as follows.
  For $n \geq 1$ and $a \in A_{n+1}$, we have
  \begin{align}\label{eq:somehow13}
    \tilde h'_{n+1}(a)(1)
    & \stackrel{\eqref{eq:somehow12}}{=} \tilde\Phi_{n+1}(a)(t_{n+1}) \\
    \intertext{and}\label{eq:somehow14}
    \tilde\phi_{n+1}(a)
    & \stackrel{\eqref{eq:somehow11}}{=}
      \tilde\Phi_{n+2}\big( \alpha_{n+1}(a)\big) (t_{n+1}).
  \end{align}
  Combining the previous two equations with \eqref{eq:somehow8} produces
  \begin{equation}\label{eq:somehow15}
    \|\tilde h'_{n+1}(a)(1) - \tilde\phi_{n+1}(a) \| < \delta_{n+1}, \qquad a \in \mathcal G_{n+1}.
  \end{equation}
  
  For $n \geq 1$, choose $\gamma_n \in (0,1)$ such that for all $s_1, s_2 \in I$ with $|s_1 - s_2| < \gamma_n$, we have the estimates
  \begin{alignat}{2}\label{eq:somehow16}
    \|k'_n(a)(s_1) - k'_n(a)(s_2)\| &< \epsilon_n, & \hspace{3ex} & a \in \mathcal F_n,
  \intertext{and}\label{eq:somehow17}
    \|\tilde h'_{n+1}(a)(s_1) - \tilde h'_{n+1}(a)(s_2) \| &<
                                                             \epsilon_n,
   & \hspace{3ex} & a \in \mathcal F_{n+1}.
  \end{alignat}
  Define a continuous function
  \begin{equation}\label{eq:somehow18}
    f_n \colon I \rightarrow I \colon s \mapsto
    \begin{cases} 
      \frac{s}{1-\gamma_n} & 0 \leq s \leq 1 - \gamma_n, \\
      1 & 1 - \gamma_n < s \leq 1,
    \end{cases}
  \end{equation}
  and note that $|f_n(s) - s| < \gamma_n$ for all $s \in I$.
	
  Define $k_n \colon A_n \rightarrow ID$ by $k_n(a)(s) = k_n'(a)(f_n(s))$ for $a \in A_n$, $s \in I$, and $n \geq 1$.
  Then $\|k_n(a) - k_n'(a)\| < \epsilon_n$ for all $a \in \mathcal F_n$ and $n \geq 1$ by the choice of $\gamma_n$.
  Lemma~\ref{lem:close-diagrams} implies that $(\underline \theta, \underline k, \underline t)$ is diagrammatic representation of $\theta$.
  We work to construct a diagrammatic representation $(\underline{\tilde\phi}, \underline{\tilde h}, \underline t)$ of $\phi$ with $\psi_*(\underline \theta, \underline k, \underline t) = \pi_*(\underline{\tilde\phi}, \underline{\tilde h}, \underline t)$.
	
  For $n \geq 1$ and $a \in A_{n+1}$, we have
  \begin{equation}
    \begin{split}
      \pi\big( \tilde h'_{n+1}(a)(1) \big)
      &\eqrel{$\overset{\eqref{eq:somehow12}}{=}$} \pi\big(
        \tilde\Phi_{n+2}\big(\alpha_{n+1}(a)\big)(t_{n+1}) \big)\\
      &\eqrel{$\overset{\eqref{eq:somehow5}}{=}$} \psi\big(
        \Theta_{n+2}\big(\alpha_{n+1}(a)\big)(t_{n+1}) \big) \\
      &\eqrel{$\overset{\eqref{eq:somehow1}}{=}$}
        \psi\big(\Theta_{n+1}(a)(t_{n+1})\big) \\
      &\eqrel{$\overset{\eqref{eq:somehow9}}{=}$} 
           \psi\big( k_n'(a)(0) \big) \\
      &\eqrel{$=$}
        \psi\big( \theta_{n+1}(a) \big) \\
      &\eqrel{$\overset{\eqref{eq:somehow12.5}}{=}$}
        \pi\big( \tilde\phi_{n+1}(a) \big),
    \end{split}
  \end{equation}
  where the unlabeled equality holds since $(\underline \theta, \underline k', \underline t)$ is a diagrammatic representation of an asymptotic morphism.
  So, $\pi \mathrm{ev}^E_1\tilde h'_{n+1} = \pi \tilde\phi_{n+1} = \psi \theta_{n+1}$.
  The choice of $\mathcal G_{n+1}$ and $\delta_{n+1}$, together with \eqref{eq:somehow15}, now implies there is a homotopy $\tilde h''_n \colon A_n \rightarrow IE$ such that
  \begin{align}
    \tilde h''_n(a)(0)
    &= \tilde h'_{n+1}\big( \alpha_n(a) \big)(1),
    &&a \in A_n, \label{sh1} \\
    \tilde h''_n(a)(1)
    &= \tilde \phi_{n+1}\big( \alpha_n(a) \big),
    &&a \in A_n, \label{sh2}\\
    \pi\big( \tilde h''_n(a)(s) \big)
    &= \psi\big( \theta_{n+1}\big( \alpha_n(a) \big) \big)
    &&a \in A_n,\ s \in I, \label{sh3}
       \shortintertext{and}
    \big\| \tilde h''_n(a)(s) -
    &\tilde h'_{n+1}\big( \alpha_n(a) \big)(1) \big\| < \epsilon_n,
    &&a \in \mathcal F_n. \label{sh4}
  \end{align}
  Define a homotopy $\tilde h_n \colon A_n \rightarrow IE$ by
  \begin{equation}\label{eq:somehow20}
    \tilde h_n(a)(s) = 
    \begin{cases} 
      \tilde h'_{n+1}\big( \alpha_n(a) \big)
      \big(f_n(s)\big), 
      & 0 \leq s \leq 1 - \gamma_n, \\[.5ex]
      \tilde h''_n(a)\big(\frac{s - (1 - \gamma_n)}{\gamma_n}\big),
      & 1 - \gamma_n < s \leq 1,
    \end{cases}
    \qquad a \in A_n,\ s \in I,
  \end{equation}
  noting that continuity follows from \eqref{sh1}.
  For $n \geq 1$ and $a \in A_n$, using \eqref{eq:somehow11} and \eqref{eq:somehow12}, we have
  \begin{align}
  	\begin{split}
        \tilde h_n(a) 
      &\eqrel{$\overset{\eqref{eq:somehow20}}{=}$} \tilde h_n'\big(\alpha_n(a)\big)(0) \\
      &\eqrel{$\overset{\eqref{eq:somehow12}}{=}$} \tilde\Phi_{n+1}\big(\alpha_n(a)\big)(t_n) \\
      &\eqrel{$\overset{\eqref{eq:somehow11}}{=}$} \tilde\phi_n(a).
    \end{split}
  \shortintertext{So $\mathrm{ev}_0^E \tilde h_n = \tilde \phi_n$.  Similarly,}
    \begin{split}
        \tilde h_n(a)(1)
      &\eqrel{$\overset{\eqref{eq:somehow20}}{=}$} \tilde h_n''(a)(1) \\
      &\eqrel{$\overset{\eqref{sh2}}{=}$} \tilde\phi_{n+1}\big(\alpha_n(a)\big),
    \end{split}
  \end{align}
  and hence $\mathrm{ev}_1^E \tilde h_n = \tilde \phi_{n+1} \alpha_n$.
  Therefore, $(\underline{\tilde \phi}, \underline{\tilde h}, \underline t) \colon (\underline A, \underline \alpha) \rightarrow E$ is a diagrammatic representation of an asymptotic morphism.
	
  We claim $\pi_*(\underline{\tilde\phi}, \underline{\tilde h}, \underline t) = \psi_*(\underline \theta, \underline k, \underline t)$.
  To this end, it suffices to show $(I\pi) \tilde h_n = (I\pi) k_n$ for all $n \geq 1$ since the equalities $\pi \tilde \phi_n = \psi \theta_n$ would then follow by evaluating these homotopies at $0$.
  Fix $n \geq 1$, $a \in A_n$, and $s \in I$.
  If $0 \leq s \leq 1 - \gamma_n$, then
  \begin{align}
    \begin{split}\label{eq:somehow21}
      \pi\big( \tilde h_n(a)(s) \big)
      &\eqrel{$\overset{\eqref{eq:somehow20}}{=}$}
        \pi\big( \tilde h_{n+1}'\big(\alpha_n(a) \big)\big( f_n(s) \big) \big) \\
      &\eqrel{$\overset{\eqref{eq:somehow12.5}}{=}$}
      \psi\big( k'_n(a) \big(f_n(s) \big) \big) \\
      &\eqrel{$=$}\psi(k_n(a)(s)),
    \end{split}
    \intertext{using the definition of $k_n$ (given just after \eqref{eq:somehow23}) in the last equality.  Further, if $1 - \gamma_n < s \leq 1$, then}
    \begin{split}\label{eq:somehow22}
      \pi\big( \tilde h_n(a)(s) \big)
      &\eqrel{$\overset{\eqref{eq:somehow20}}{=}$}
        \pi\Big( \tilde h_n''(a)\Big(\frac{s - (1 - \gamma_n)}{\gamma_n}\Big)\Big) \\
      &\eqrel{$\overset{\eqref{sh3}}{=}$}
        \psi\big( \theta_{n+1}\big(\alpha_n(a) \big) \big) \\
      &\eqrel{$=$} \psi\big( k_n(a)(1) \big) \\
      &\eqrel{$\overset{\eqref{eq:somehow18}}{=}$}
        \psi\big( k_n(a) \big( f_n(s) \big) \big),
    \end{split}
  \end{align}
  where the third equality holds since $(\underline \theta, \underline k, \underline t)$ is a diagrammatic representation of an asymptotic morphism. Therefore, the claim holds.

  It remains to show that the diagrammatic representation $(\underline{\tilde \phi}, \underline{\tilde h}, \underline t)$ induces the asymptotic morphism $\tilde \phi$.  Let $\tilde\phi' \colon A \xrightarrow\approx E$ be the asymptotic morphism induced by $(\underline{\tilde \phi}, \underline{\tilde h}, \underline t)$.
  We must show $\tilde \phi \cong \tilde \phi'$.
  To this end, it suffices to show $\tilde \phi_{\rm as} = \tilde \phi'_{\rm as} \colon A \rightarrow E_{\rm as}$.
	
  We start by showing that for all $n \geq 1$,
  \begin{align}\label{eq:somehow23}
    \big\| \tilde h_n(a)(s) - \tilde h'_{n+1}\big( \alpha_n(a)
    \big)(s) \big\| 
    &< 2\epsilon_n,
    && a \in \mathcal F_n,\ s \in I.
       \intertext{For $s \in I$, we have $|f_n(s) - s| < \gamma_n$ as noted just after \eqref{eq:somehow18}. Since $\alpha_n(\mathcal F_n) \subseteq \mathcal F_{n+1}$, \eqref{eq:somehow17} implies}\label{eq:somehow24}
   \big\| \tilde h'_{n+1} \big( \alpha_n(a) \big) \big( f(s) \big) - \tilde h'_{n+1}\big( \alpha_n(a)\big)(s) \big\|
    &< \epsilon_n, \
    &&a \in \mathcal F_n,\ s \in I.
  \end{align}
  When $0 \leq s \leq 1 - \gamma_n$, \eqref{eq:somehow23} follows (with $\epsilon_n$ in place of $2 \epsilon_n$), using \eqref{eq:somehow20} to identify $\tilde h'_{n+1}(\alpha_n(a))(f(s))$ with $\tilde h_n(a)(s)$.
  When $1 - \gamma_n < s \leq 1$, we have $f(s)= 1$ by \eqref{eq:somehow18}, so \eqref{eq:somehow23} follows from combining \eqref{sh4} and \eqref{eq:somehow24}, using \eqref{eq:somehow20} to compute $\tilde h_n(a)(s)$.
  This completes the proof of \eqref{eq:somehow23}.
	
  For $n \geq 1$, let $\epsilon'_n = \sum_{m=n}^\infty \epsilon_m$ and note that $\lim_n \epsilon'_n = 0$, since $\sum_{m=1}^\infty \epsilon_m < \infty$. We will show that
  \begin{equation}\label{eq:somehow25}
    \big\|\tilde  \phi_{\rm as} \big(\alpha_{\infty, n}(a)\big) - \tilde\phi'_{\rm as}\big(\alpha_{\infty, n}(a) \big) \big\| \leq 2\epsilon_n + \epsilon'_n, \qquad a \in \mathcal F_n,\ n \geq 1.
  \end{equation}
  Since $\bigcup_{n=1}^\infty \alpha_{\infty, n}(\mathcal F_n)$ is dense in $A$, $\lim_n \epsilon_n = 0$, and $\lim_n \epsilon'_n = 0$, \eqref{eq:somehow25} will prove $\tilde\phi_{\rm as} = \tilde\phi'_{\rm as}$, completing the proof.
  
  Fix $n \geq 1$.
  Define $\tilde\Phi'_n \colon A_n \rightarrow C_b(\mathbb R_{\geq m_{n+2}}, E)$ by
  \begin{equation}\label{eq:somehow26}
    \tilde\Phi'_n(a)(t) = 
    \begin{cases}
      \tilde\phi_n(a) & m_{n + 2} \leq t < t_n \\[.5ex]
      \tilde h_m\big(\alpha_{m, n}(a)\big)(\frac{t - t_m}{t_{m+1} - t_m}) & t_m \leq t < t_{m+1},\ m \geq n,
    \end{cases}
  \end{equation}
  for $a \in A_n$ and $t \in \mathbb R_{\geq m_{n+2}}$.  Then
  \begin{equation}\label{eq:somehow27}
    \tilde\phi'_{\rm as} \alpha_{\infty, n} = \rho^E_{\infty, n+2} \tilde\Phi'_n \qquad \text{and} \qquad   \tilde \phi_{\rm as} \alpha_{\infty, n} \overset{\eqref{eq:somehow18}}{=} \rho^E_{\infty, n+2} \tilde\Phi_{n+1}\alpha_n,
  \end{equation}
  We will prove the following bound, which readily implies \eqref{eq:somehow25}:
  \begin{equation}\label{eq:somehow28}
    \big\| \tilde\Phi_{n+1}\big( \alpha_n(a) \big)(t) - \tilde\Phi'_n(a)(t) \big\| < 2\epsilon_n + \epsilon'_n, \qquad a \in \mathcal F_n,\ t \geq t_n.
  \end{equation}
  
  With $n \geq 1$ still fixed, let $a \in \mathcal F_n$ and $t \geq t_n$ be given.
  Let $m \geq n$ be such that $t_m \leq t < t_{m+1}$.
  Since $\alpha_n(\mathcal F_n) \subseteq \mathcal G_{n+1}$ and $\epsilon_n > \delta_{n+1}$, applying \eqref{eq:somehow8} inductively yields
  \begin{equation}\label{eq:somehow29}
    \big\| \tilde\Phi_{n+1}\big( \alpha_n(a) \big)(t) - \tilde\Phi_{m+1}\big(\alpha_{m+1, n}(a) \big)(t) \big\| < \epsilon_n'.
  \end{equation}
  Note that
  \begin{align}
    \tilde\Phi_{m+1}\big( \alpha_{m+1,n}(a) \big)(t)
    &\eqrel{$\overset{\eqref{eq:somehow12}}{=}$}
      \tilde h'_{m+1}\big( \alpha_{m+1, n}(a) \big)\Big(
      \frac{t - t_m}{t_{m+1}-t_m}\Big)
      \intertext{and}
      \tilde\Phi'_n(a)(t)
    &\eqrel{$\overset{\eqref{eq:somehow26}}{=}$}
      \tilde h_m\big( \alpha_{m, n}(a) \big)\Big(
      \frac{t - t_m}{t_{m+1}-t_m}\Big)
  \end{align}
  Since $\alpha_{m, n}(a) \in \mathcal F_m$ and $\epsilon_m \leq \epsilon_n$, the previous two equalities combine with \eqref{eq:somehow23} and \eqref{eq:somehow29} to produce \eqref{eq:somehow28}, which (finally) completes the proof.
\end{proof}

The rest of the proof closely follows the argument sketched in Section~\ref{sec:ahlp}.  One just needs to take care with the index shift arising from the relative homotopy lifting property and the several identifications being made in the sketch.  The following lemma is deigned to help with the latter.

\begin{lemma}
  \label{lem:X}
  Suppose $A_0$, $A$, $B$, and $E$ are $C^*$-algebras and $\alpha_0 \colon A_0 \rightarrow A$ and $\pi \colon E \rightarrow B$ are $^*$-homomorphisms with $\pi$ surjective.
  Suppose further that $(\alpha, I\pi)$ satisfies the homotopy lifting property.
  Let $X \subseteq I^2$ be such that there is a homeomorphism $f \colon I^2 \rightarrow I^2$ such that $f(X) = I \times \{0\}$.
  Then the diagram completion problem
  \begin{equation}\label{eq:X}
    \begin{tikzcd}
      A  \arrow[bend left]{ddrrr}{\tilde\nu} \arrow[bend right]{dddrr}[swap]{\mu}&[-10pt] &[-10pt] &[15pt] \\[-10pt]
      & A_0 \arrow{ul}[swap]{\alpha} \arrow[dashed]{dr}{\tilde\mu} & & \\[-10pt]
      & & I^2E \arrow{r}{\rho^E_X} \arrow{d}[swap]{I^2\pi} & XE \arrow{d}{X\pi} \\[15pt]
      & & I^2B \arrow{r}[swap]{\rho^B_X} & XB 
    \end{tikzcd}
  \end{equation}
  always has a solution $\tilde\mu$, where $\rho^B_X$ and $\rho^E_X$ are the restriction maps.
\end{lemma}

\begin{proof}
  For each C$^*$-algebra $D$, define isomorphisms \begin{equation}
  	\tilde\eta_f^D \colon I^2 D \xrightarrow\cong I(ID) \qquad \text{and} \qquad \eta_f^D \colon XD \xrightarrow\cong ID
  \end{equation}
  by
  \begin{align}
  	\tilde\eta_f^D(g)(s_1)(s_2) &= g\big(f^{-1}(s_2, s_1)\big), &&g \in I^2D,\ s_1, s_2 \in I,
  \intertext{and}
    \eta_f^D(g)(s) &= g\big(f^{-1}(s, 0)\big), && g \in XD,\ s \in I.
  \end{align}
  Then the diagram
  \begin{equation}
  	\begin{tikzcd}
      & I(IE) \arrow{rr}{\mathrm{ev}_0^{IE}} \arrow{dd}[pos = .2]{I(I\pi)} & & IE \arrow{dd}{I\pi} \\
      I^2 E \arrow[crossing over]{rr}[pos = .8,swap]{\rho_X^E} \arrow{ur}{\tilde\eta_f^E} \arrow{dd}[swap]{I^2\pi} & & XE \arrow{ur}[swap, pos = .4]{\eta_f^E} & \\
      & I(IB) \arrow{rr}[pos = .2, swap]{\mathrm{ev}_0^{IB}} & & IB \\
      I^2B \arrow{rr}{\rho_X^B} \arrow{ur}{\tilde\eta_f^B} & & XB \arrow{ur}[swap]{\eta_f^B} \arrow[leftarrow, crossing over]{uu}[swap, pos = .8]{X\pi} &		
    \end{tikzcd}
  \end{equation}
  commutes.  Indeed, the faces involving $\pi$ commute by the naturality of $\eta_f$, $\tilde\eta_f$, $\rho_X$, and $\mathrm{ev}_0$.  To see the commutativity of the top and bottom faces, note that if $D$ is a C$^*$-algebra $g \in I^2D$, and $s \in I$, then
  \begin{align}
  	\begin{split}
  	    \mathrm{ev}_0^{ID}\big(\tilde \eta_f^D(g)\big)(s)
  	  &= \tilde \eta_f^D(g)(0)(s) \\
  	  &= g\big( f^{-1}(s, 0) \big) \\
  	  &= \rho_X^D(g)\big(f^{-1}(s, 0)\big) \\
  	  &= \eta_f^D\big(\rho_X^D(g)\big)(s).
    \end{split}
  \end{align}

  Since $(\alpha, I\pi)$ satisfies the homotopy lifting property, there is a $^*$-homomorphism $\tilde\mu' \colon A_0 \rightarrow I(IE)$ such that
  \begin{equation}
  	\big(I(I\pi)\big) \tilde\mu' = \tilde\eta_f^B \mu \alpha \qquad \text{and} \qquad \mathrm{ev}_0^{IE} \tilde \mu' = \eta_f^E \tilde\nu \alpha.
  \end{equation}
  Define $\tilde \mu = (\tilde \eta_f^E)^{-1} \tilde\mu'$.
\end{proof}

We are now ready to prove the main asymptotic homotopy lifting theorem (Theorem~\ref{thm:ahlp}).  As already noted in Section~\ref{sec:ahlp}, this implies Theorem~\ref{thm:ahlp-asp} and~\ref{thm:ahlp-qd} from the introduction.

\begin{proof}[Proof of Theorem~\ref{thm:ahlp}]
  Suppose $\theta \colon A \xrightarrow\approx IB$ and $\tilde\phi \colon A \xrightarrow\approx E$ are asymptotic morphisms such that $\mathrm{ev}_0^B \theta \cong \pi \tilde\phi$.
  By Lemma~\ref{lem:somehow}, there are diagrammatic representations $(\underline \theta, \underline k, \underline t) \colon (\underline A, \underline \alpha) \rightarrow B$ and $(\underline{\tilde\phi}, \underline{\tilde h}, \underline t) \colon (\underline A, \underline \alpha) \rightarrow E$ of $\theta$ and $\tilde\phi$, respectively, such that
  \begin{equation}
    \label{eq:ahlp.1}
    (\mathrm{ev}_0^B)_*(\underline \theta, \underline k, \underline t) = \pi_*(\underline{\tilde \phi}, \underline{\tilde h}, \underline t).
  \end{equation}
	
  Note that $\mathrm{ev_0^B} \theta_{n+2} = \pi \tilde\phi_{n+2}$.
  As $(\alpha_{n+1}, \pi)$ satisfies the homotopy lifting property by hypothesis, there is a $^*$-homomorphism $\tilde\theta_{n+1}' \colon A_{n+1} \rightarrow IE$ such that
  \begin{equation}
    \label{eq:ahlp.2}
    \mathrm{ev}_0^E \tilde\theta'_{n+1} = \tilde\phi_{n+2} \alpha_{n+1} \quad\text{and}\quad (I\pi) \tilde\theta'_{n+1} = \theta_{n+2}\alpha_{n+1}.
  \end{equation}
  Define
  \begin{equation}
    \label{eq:ahlp.3}
    X = \big\{ (s_1, s_2) \in I \times I : s_1 = 0,\ s_1 = 1,\ \text{or}\ s_2 = 0 \big\}
  \end{equation}
  and note that there is a homeomorphism $f \colon I^2 \rightarrow I^2$ such that $f(X) = I \times \{0\}$.
  Then define $\tilde \nu_{n+1} \colon A_{n+1} \rightarrow XE$ by
  \begin{equation}
    \label{eq:ahlp.4}
    \tilde \nu_{n+1}(a)(s_1, s_2) =
    \begin{cases}
      \tilde\theta'_{n+1}(a)(s_2) & s_1 = 0 \\
      \tilde\theta'_{n+2}\big(\alpha_{n+1}(a)\big)(s_2) & s_1 = 1 \\
      \tilde h_{n+2}\big(\alpha_{n+1}(a)\big)(s_1) & s_2 = 0
    \end{cases}
  \end{equation}
  and define $\mu_{n+1} \colon A_{n+1} \rightarrow I^2B$ by $\mu_{n+1}(a)(s_1, s_2) = k_{n+2}(\alpha_{n+1}(a))(s_1)(s_2)$.
	
  Let $n \geq 1$ and $a \in A_{n+1}$ be given. For $s_2 \in I$, the definition of $\mu_{n+1}$, the fact that $(\underline{\theta}, \underline{k}, \underline{t})$ is a diagrammatic representation of $\theta$, \eqref{eq:ahlp.2}, and \eqref{eq:ahlp.4} give
  \begin{align}
    \begin{split}
      \mu_{n+1}(a)(0, s_2) &= k_{n+2}\big(\alpha_{n+1}(a)\big)(0)(s_2) \\ &=
      \theta_{n+2}\big(\alpha_{n+1}(a)\big)(s_2) \\ &= \pi\big(\tilde\theta_{n+1}'(a)(s_2)\big) \\ &= \pi\big(\tilde\nu_{n+1}(a)(0, s_2)\big),
    \end{split}
    \intertext{and, for the same reasons,}
    \begin{split}
      \mu_{n+1}(a)(1,s_2) 
      &= k_{n+2}\big(\alpha_{n+1}(a)\big)(1)(s_2) \\ 
      &= \theta_{n+3}\big(\alpha_{n+3, n+1}(a)\big)(s_2)
      \\ 
      &= \pi\big(\tilde\theta'_{n+2}\big(\alpha_{n+1}(a)\big)(s_2)\big) \\ 
      &= \pi\big(\tilde\nu_{n+1}(a)(1,s_2)\big)
    \end{split}
    \intertext{Moreover, the definition of $\mu_{n+1}$, \eqref{eq:ahlp.1}, and \eqref{eq:ahlp.4} imply that for all $s_1 \in I$, we have}
    \begin{split}
      \mu_{n + 1}(a)(s_1, 0) &= k_{n+2}\big(\alpha_{n+1}(a)\big)(s_1)(0) \\ &=  \pi\big(\tilde h_{n+2}\big(\alpha_{n+1}(a)\big)(s_1)\big) \\ &= \pi\big(\tilde \nu_{n+1}(a)(s_1, 0)\big).
    \end{split}
  \end{align}
  The three computations show that $\rho^B_X \mu_{n+1} = (X\pi)\tilde\nu_{n+1}$.
	
  By Lemma~\ref{lem:X}, for all $n \geq 1$, there is a $^*$-homomorphism $\tilde\mu_n \colon A_n \rightarrow I^2E$ such that
  \begin{equation}
    \label{eq:ahlp.5}
    \rho^E_X \tilde \mu_n = \tilde \nu_{n+1} \alpha_n \qquad\text{and}\qquad (I^2\pi) \tilde\mu_n = \mu_{n+1} \alpha_n.
  \end{equation}
  For $n \geq 1$, define $\tilde \theta_n \colon A_n \rightarrow IE$ and $\tilde k_n \colon A_n \rightarrow I(IE)$ by
  \begin{align}
    \tilde \theta_n(a)(s) &= \tilde\theta'_{n+1}\big(\alpha_n(a)
    \big)(s), &&a \in A_n, s \in I,\label{eq:ahlp.6}
                                                                     \intertext{and}
                                                                     \tilde k_n(a)(s_1)(s_2) &= \tilde\mu_n(a)(s_1, s_2), &&a \in A_n,\ s_1, s_2 \in I.\label{eq:ahlp.7}
  \end{align}
  Fix $n \geq 1$.  Then \eqref{eq:ahlp.6}, \eqref{eq:ahlp.5}, \eqref{eq:ahlp.4}, and \eqref{eq:ahlp.6} give
  \begin{align}
    \begin{split}
      \tilde k_n(a)(0)(s) 
      &= \tilde\mu_n(a)(0, s) \\
      &= \tilde \nu_{n+1}\big(\alpha_n(a)\big)(0, s) \\
      &= \tilde\theta'\big(\alpha_n(a)\big)(s) \\
      &= \tilde\theta_n(a)(s)
    \end{split}
    \intertext{for all $a \in A_n$ and $s \in I$.
    Thus $(I\mathrm{ev}^{IE}_0) \tilde k = \tilde \theta_n$.
    Further, for the same reasons,}
	\begin{split}
          \tilde k_n(a)(1)(s) 
          &= \tilde \mu_n(a)(1, s) \\
          &= \tilde \nu_{n+1}\big(\alpha_n(a)\big)(1,s) \\
          &= \tilde \theta'_{n+2}\big(\alpha_{n+2, n}(a)\big)(s) \\
          &= \tilde \theta_{n+1}\big(\alpha_n(a)\big)(s)
	\end{split}
  \end{align}
  for all $a \in A_n$ and $s \in I$.
  Therefore, $(I\mathrm{ev}_1^{IE})\tilde k_n = \tilde\theta_{n+1} \alpha_n$.
  Let $\underline t_{+2}$ denote the sequence $(t_{n+2})_{n=1}^\infty$.
  Then $(\underline{\tilde \theta}, \underline{\tilde k}, \underline t_{+2}) \colon (\underline A, \underline \alpha) \rightarrow IE$ is a diagrammatic representation for an asymptotic morphism $\tilde \theta \colon A \xrightarrow\approx IE$.
	
  For $n \geq 1$, $a \in A_n$, and $s_1, s_2 \in I$, by \eqref{eq:ahlp.7}, \eqref{eq:ahlp.5}, and the definition of $\mu_{n+1}$, we have
  \begin{equation}\label{eq:ahlp.8}
    \begin{split}
      \pi\big(\tilde k_n(a)(s_1)(s_2)\big) 
      &= \pi\big(\tilde \mu_n(a)(s_1, s_2)\big) \\
      &= \mu_{n+1}\big(\alpha_n(a)\big)(s_1, s_2) \\
      &= k_{n+2}\big(\alpha_{n+2, n}(a)\big)(s_1)(s_2),
    \end{split}
  \end{equation}
  and hence $(I(I\pi)) \tilde k_n = k_{n+2}\alpha_{n+2, n}$.
  We also have $(I\pi) \tilde \theta_n = \theta_{n+2} \alpha_{n+2, n}$ by applying \eqref{eq:ahlp.8} with $s_1 = 0$.
  Therefore, $\pi_*(\underline{\tilde \theta}, \underline{\tilde k}, \underline t_{+2}) = (\underline \theta, \underline k, \underline t)_{+2}$.
  By Lemma~\ref{lem:shift}, $\pi \tilde \theta \cong \theta$.
	
  Similarly, note that for all $n \geq 1$, $a \in A_n$, and $s \in I$, we have
  \begin{equation}\label{eq:ahlp.9}
    \begin{split}
      (I\mathrm{ev}^E_0)\big(\tilde k_n(a)\big)(s)
      &= \tilde k_n(a)(s)(0) \\
      &= \tilde \mu_n(a)(s, 0) \\
      &= \tilde \nu_{n+1}\big(\alpha_n(a)\big)(s, 0) \\
      &= \tilde h_{n+2}\big(\alpha_{n+2, n}(a)\big)(s),
    \end{split}
  \end{equation}
  and hence $(I\mathrm{ev}^{E}_0) \tilde k_n = \tilde h_{n+2} \alpha_{n+2, n}$.
  Taking $s_1 = 0$ in \eqref{eq:ahlp.9} implies that $\mathrm{ev}^E_0 \tilde \theta_n = \tilde \phi_{n+2}\alpha_{n+2, n}$.
  Thus $(\mathrm{ev}^E_0)_*(\underline{\tilde \theta}, \underline{\tilde k}, \underline t_{+2}) = (\underline{\tilde \phi}, \underline{\tilde h}, \underline t)_{+2}$.
  Another application of Lemma~\ref{lem:shift} implies that $\mathrm{ev}^E_0 \tilde\theta \cong \tilde\phi$, which completes the proof.
\end{proof}


\end{document}